\documentclass[10pt,a4paper]{amsart} 
\usepackage{geometry}
\usepackage{tabls}
\usepackage{url}
\usepackage[utf8]{inputenc}
\usepackage{amsfonts}
\usepackage{amssymb}
\usepackage{mathrsfs}
\usepackage{hyperref}
\hypersetup{hidelinks,
	colorlinks=true,
	allcolors=black,
	pdfstartview=Fit,
	breaklinks=true}
\usepackage{amsmath}
\usepackage{amscd}
\usepackage{fancyhdr}
\usepackage{multirow}
\usepackage{enumerate}
\usepackage{tikz}
\usepackage{tikz-qtree}
\usepackage{paralist}
\usepackage{xypic}
\usepackage{graphicx}
\usepackage{cite}
\usepackage{lipsum}
\usepackage{footmisc}
\usepackage[all,cmtip]{xy}

\allowdisplaybreaks

\numberwithin{equation}{section}
\theoremstyle{plain}
\newtheorem{theorem}{Theorem}[section]
\newtheorem{lemma}[theorem]{Lemma}
\newtheorem{proposition}[theorem]{Proposition}
\newtheorem{conjecture}[theorem]{Conjecture}
\newtheorem{remark}[theorem]{Remark}
\newtheorem{definition}[theorem]{Definition}
\newtheorem{corollary}[theorem]{Corollary}
\numberwithin{equation}{section}

\def\a{\alpha}
\def\b{\beta}
\def\d{\delta}

\def\g{\gamma}

\def\n{\vert}

\def\mfa{\mathfrak{a}}
\def\mfb{\mathfrak{b}}

\def\ni{\noindent}

\def\Re{{\rm Re\/}}
\def\Im{{\rm Im\/}}

\def\R{\mathbb{R} }
\def\C{\mathbb{C}}

\def\N{\mathbb{N}}
\def\Z{\mathbb{Z}}

\begin{document}
	
	\title{$K_0$ groups of noncommutative $\mathbb{R}^{2n}$}

	
	\author{Ren Guan}
	\address{School of Mathematics and Statistics, 
		Jiangsu Normal University, Xuzhou 221100, China}
	\email{guanren@jsnu.edu.cn}

	\begin{abstract}
		In this paper we show that the $K_0$ groups of noncommutative $\mathbb{R}^{2n}$ are $\mathbb{Z}$ for $\forall n\in\mathbb{N}^*$ and make an approach to the calculation of the smooth case, which will bring many new sequence problems relating to binomial numbers.
		
	\end{abstract}
	
	
	
	\maketitle
	
	\tableofcontents
	
	\section{Introduction}
	
	For  $n\in\N$, the noncommutative $\R^{2n}$, or noncommutative flat space-time, denoted by $A(\mathbb{R}^{2n}_\Theta)$, which naturally arised from noncommutative field theory\cite{DN,GV,GW}, are another important class of noncommutative differentiable manifolds besides noncommutative tori\cite{DN,Re}.  When regarded as a subalgebra of the bounded operators $B(H)$ on a separable Hilbert space $H$, $A(\mathbb{R}^{2n}_\Theta)$ can be described as a $*$-algebra  generated by $2n$ self-adjoint operators $x_1,x_2,\ldots,x_{2n}$ satisfying
	\begin{equation}\label{r2n}
		[x_p,x_q]=\begin{cases}
			-i\theta_{pq}, &2|q,~p=q-1\text{ or }2|p,~q=p-1,\\
			-i\theta_{pq}=0, &\text{other cases,}
		\end{cases}
	\end{equation}
	where $-\theta_{qp}=\theta_{pq}>0$ for $2|q$ and $p=q-1$, $1\leq p,q\leq 2n$, $i=\sqrt{-1}$, $\Theta:=\{\theta_{pq}\}_{1\leq p,q\leq2n}$ is a real skew-symmetric matrix. The relation \eqref{r2n} ensures that any product of $x_1,x_2,\ldots,x_{2n}$ can be rearranged to a finite sum of the form
	\begin{equation}
		\sum a_{p_1,p_2,\ldots,p_{2n}}x_1^{p_1}x_2^{p_2}\ldots x_{2n}^{p_{2n}}
	\end{equation}
	with $(p_1,p_2,\ldots,p_{2n})\in\N^{2n}$, $a_{p_1,p_2,\ldots,p_{2n}}\in\C$ and we set $x^0_i=I$, the identity operator of $B(H)$. For example,
	\begin{equation}\begin{aligned}
			x_4x_2x_3x_1&=x_2x_1x_4x_3=(x_1x_2+i\theta_{12})(x_3x_4+i\theta_{34})\\
			&=x_1x_2x_3x_4+i\theta_{34}x_1x_2+i\theta_{12}x_3x_4-\theta_{12}\theta_{34}.
		\end{aligned}
	\end{equation}
	Mimick the definition of the smooth noncommutative tori\cite{EH}, the smooth version $A^\infty(\mathbb{R}^{2n}_\Theta)$ of noncommutative $\R^{2n}$ is the $C^*$-algebra of formal series
	\begin{equation}
		\sum a_{p_1,p_2,\ldots,p_{2n}}x_1^{p_1}x_2^{p_2}\ldots x_{2n}^{p_{2n}}
	\end{equation}
	where the coefficient function $\N^{2n}\owns(p_1,p_2,\ldots,p_{2n})\mapsto a_{p_1,p_2,\ldots,p_{2n}}\in\C$ belongs to
	the Schwartz space $\mathcal{S}(\N^{2n})$, i.e., there is a constant $C_r$ for every $r\geq1$ such that 
	\begin{equation}
		\sup_{(p_1,p_2,\ldots,p_{2n})\in\N^{2n}}\left(1+\sum_{k=1}^{2n}p_k^2\right)^r a_{p_1,p_2,\ldots,p_{2n}}<C_r.
	\end{equation}
	
	For a $*$-algebra $A$, a \emph{projector}(or \emph{projection}) $p$ of $A$ is a matrix with entries in $A$ and satisfies $p^2=p=p^*$. We denote by $P(A)$ the set of projectors of $A$. For any two projectors $p,q\in P(A)$, define
	\begin{equation}\label{kkk}
		p+q:=\begin{pmatrix}
			p&0\\
			0&q
		\end{pmatrix},
	\end{equation}
	and call $p,q$ \emph{equivalent}, $p\sim q$, if there is a unitary $u\in M_n(A)$ for some suitable $n\in\N$ such that 
	\begin{equation}
		\begin{pmatrix}
			p&0\\
			0&0
		\end{pmatrix}=u\begin{pmatrix}
			q&0\\
			0&0
		\end{pmatrix}u^*.
	\end{equation}
	Then $P(A)/\sim$  forms a semigroup under the operation \eqref{kkk}. The $K_0$ \emph{group} $K_0(A)$ of $A$ is defined as the Grothendieck group of the (direct sum) semigroup of isomorphism classes of finitely generated projective right modules over $A$\cite[Section 3.2]{GVF}\cite{R}, or equivalently,  the Grothendieck group of $P(A)/\sim$. Computing the $K_0$ groups of an algebra is a natural question, in general it's not easy, even for commutative one. 
	
	$K_0$ groups are important for noncommutative geometry, Exel shows that Morita equivalent algebras have isomorphic $K_0$ groups\cite{Ex}, and the famous Baum-Connes conjecture\cite{BC} is related to the $K$-theory groups $K_j(C^*_r(G))$ where $j=0,1$ of the reduced $C^*$-algebra $C^*_r(G)$ of a locally compact, Hausdorff and second countable group $G$. See \cite{HLS,MY,V} for significant development in Baum-Connes conjecture, which also provide examples of calculating $K_0$ groups, and \cite{HM} for the calculation of relative algebraic $K$-theory of a truncated polynomial algebra over a perfect field $k$ of positive characteristic $p$, \cite{HML} for division algebras over local fields, etc.
	
	In\cite{RS}, Rieffel and Schwarz give a equivalent condition for Morita equivalence of two noncommutative tori, and Elliott and Hanfeng Li discuss the smooth case\cite{EH}. Notably, Pimsner and Voiculescu construct a hexagonal exact sequence linking the $K$-groups of an algebra $A$ and of  $A\rtimes_{\a}\Z$, the crossed products of $A$ by $\Z$\cite{PV}. By applying this result and considering the noncommutative tori $\mathbb{T}^n_{\Theta}$ as an iterated crossed product $\C\rtimes_{\a_1}\Z\rtimes_{\a_2}\Z\rtimes\ldots\rtimes_{\a_n}\Z$ (cf. \cite[Proposition 12.8]{GVF}), Elliott shows that $K_0(\mathbb{T}^n_{\Theta})=K_1(\mathbb{T}^n_{\Theta})=\Z^{2^{n-1}}$, which are independent of the choices of $\Theta$.
	
	In this paper, We first compute $K_0(A(\mathbb{R}^{2n}_\Theta))$, and then make an approach to the computation of the smooth cases $K_0(A^\infty(\mathbb{R}^{2n}_\Theta))$. By the definition of the $K_0$ group, we need to characterize the projectors of $A(\mathbb{R}^{2n}_\Theta)$ and $A^\infty(\mathbb{R}^{2n}_\Theta)$ under the unitary equivalence respectively. Our first main result is
	\begin{theorem}\label{m1}
		$K_0(A(\mathbb{R}^{2n}_\Theta))=\Z$ for any $n\in\N^*$ and $\Theta$.
	\end{theorem}
	But for the smooth case $A^\infty(\mathbb{R}^{2n}_\Theta)$, even the characterization of the  projective elements of $A^\infty(\mathbb{R}^{2n}_\Theta)$ is a highly nontrivial problem. In this paper we consider the classification of a special class of projectors and the following  conclusion is of great importance.
	\begin{conjecture}\label{cjj}
		Let $\{b_n\}_{n\geq0}$ be a complex number sequence such that for some integer $l\geq1$, $b_{n}b_{n+l}=0$ for $\forall n\in\N$ and define 
		$$a_n:=\sum_{k=0}^n(-1)^{n-k}\binom{n}{k}b_k.$$
		If $\lim_{n\to\infty}a_n=0$, then $b_n=0$ for all $n$.
	\end{conjecture}
	Currently I can't prove it or construct a counterexample, so I state it as a conjecture.  We have the following theorem, which is proved in section \ref{scp}.
	\begin{theorem}\label{m2}
		If Conjecture \ref{cjj} is true and $$\mathcal{P}:=\sum_{p,q=0}^\infty a_{p,q}x^py^q\in P(A^\infty(\R_\theta^2))$$ such that $a_{p,q}=0$ for all pairs $(p,q)$ satisfying $\n p-q\n>k$ for some $k\in\N$, then $\mathcal{P}=0$ or $1$. 
	\end{theorem}
	\begin{remark}
		Conjecture \ref{cjj} does not hold if we weaken the constraint $b_nb_{n+l}=0$. The following counterexample is given by Noam D. Elkies: Let
		\begin{equation}
			b_n=\frac{\mathrm{sgn}(\sin\frac{n\pi}{3})}{2^n},
		\end{equation}
		then $b_n=0$ when $3|n$, and correspondingly
		\begin{equation}
			a_n\ll \left(\frac{3}{4}\right)^\frac{n}{2}.
		\end{equation}
		In fact, by applying Noam D. Elkies' method, for any $q\geq3$, we can construct a sequence $\{b_n\}_{n\geq0}$ such that $\lim_{n\to\infty}a_n=0$ and $b_n=0$ wherever $q\n n$. 
		
		By applying the Fabry Gap Theorem \cite{B}, Alexandre Eremenko shows that $b_n=O(\epsilon^n)$ for $\forall\epsilon>0$. And if $\n a_n\n$ tends to zero with geometric speed, that is $\n a_n\n=O(\delta^n)$ for some $\d\in(0,1)$, we can conclude that $b_n=0$ for all $n$. See  \href{https://mathoverflow.net/questions/425948/a-number-sequence-problem-involving-binomial-transform}{https://mathoverflow.net/questions/425948/a-number-sequence-problem-involving-binomial-transform} for details.
	\end{remark}
	We also consider the higher-dimensional analogues $P_k(A^\infty(\R_\Theta^{2n}))$ in section \ref{scp}, see Definition \ref{pkn}.  We first consider $P_0(A^\infty(\R_\Theta^{4}))$, then make some primary discussion to $P_k(A^\infty(\R_\Theta^{4}))$ and left the $n\geq3$ cases for future. Naturally we conjecture that:
	\begin{conjecture}\label{conjpk}
		$P_k(A^\infty(\R_\Theta^{2n}))=\{0,1\}$ for $\forall k\geq0$, $\forall n\geq1$ and $\Theta$.
	\end{conjecture}
	The ultimate goal is the characterization of $P(A^\infty(\R_\Theta^{2n}))$ where $n\in\N^*$,  which will tell us  what $K_0(A^\infty(\R_\Theta^{2n}))$ looks like,  is also left for the future. Mimiking Theorem \ref{m1}, we propose the following conjecture:
	\begin{conjecture}\label{k0nTh}
		$K_0(A^\infty(\mathbb{R}^{2n}_\Theta))=\Z$ for $\forall n\in\N^*$ and $\Theta$.
	\end{conjecture}
	
	In section \ref{sfa}, we give a characterization of self-adjoint elements of $A^\infty(\R_\theta^2)$. If $\sum_{p,q=0}^\infty a_{p,q}x^py^q\in A^\infty(\R_\theta^2)$ is self-adjoint, then the imaginary parts of the coefficients $a_{p,q}$ can be expressed by their real parts, see Theorem \ref{sa} for details. Projectors are always self-adjoint, so  we state here for future use. See \cite{MY2} for the similar work for the pseudo-unitary group $\bold{U}(p,p)$.
	
	Through the whole paper we set all terms containing $n!$ with $n<0$ to 0. For example, $$\frac{n!}{(n-m+p)!}\binom{m}{p}a^{m-p}x^py^{n-m+p}:=0\text{ if }n-m+p<0.$$
	And unless stated otherwise, $i=\sqrt{-1}$, $j,k,l,m,n,p,q,r,s,t\in\N$, $a_{p_1,p_2,\ldots,p_{2n}}\in\C$, and $a_{p_1,p_2,\ldots,p_{2n}}=0$ if one of ${p_1,p_2,\ldots,p_{2n}}$ is less than zero. $\Theta:=\{\theta_{pq}\}_{1\leq p,q\leq2n}$ is a real skew-symmetric matrix where $-\theta_{qp}=\theta_{pq}>0$ for $2|q$, $p=q-1$ and $\theta_{pq}=0$ for other cases, $1\leq p,q\leq 2n$. Although some symbols ($\mfa,\mfb,\mathcal{P}$, etc.) are repeatly used, the readers won't be confused.
	
	\ni\textbf{Acknowledgement.} 
	This research is partially supported by NSFC grants 12201255.
	
	\section{The nonsmooth case}
	
	Let's start from $A(\mathbb{R}^{2}_\theta)$, which is generated by two self-adjoint operators $x$ and $y$ satisfying
	\begin{equation}\label{re}
		[x,y]=-i\theta
	\end{equation}
	for a real number $\theta\geq0$. For convenience, in the following we denote $a:=i\theta$, so \eqref{re} is equivalent to
	\begin{equation}\label{a}
		yx=xy+a.
	\end{equation}
	The *-operation on an element of $A(\mathbb{R}^2_\theta)$ is  taking its adjoint as an operator in  $B(H)$. Every element $T\in A(\mathbb{R}^2_\theta)$ can be written as 
	\begin{equation}\label{ele}
		T=\sum_{k=0}^n\sum_{p+q=k}a_{p,q}x^py^q
	\end{equation}
	for some integer $n\in\N$ and $a_{p,q}\in\C$ for all $p$ and $q$.  For example, 
	\begin{align*}
		y^2x&=y(xy+a)=(xy+a)y+ay=2ay+xy^2,\\
		y^2x^2&=(2ay+xy^2)x=2a(xy+a)+x(2ay+xy^2)=2a^2+4axy+x^2y^2.
	\end{align*} 
	The smallest $n$ in \eqref{ele} is called the \emph{degree} of $T$, denoted by $\deg(T)$, which means for such $n$, at least one of the complex numbers  $\{a_{0,n},a_{1,n-1},\ldots,a_{n,0}\}$ is nonzero. And we call
	\begin{equation}
		\sigma(T):=\sum_{p+q=n}a_{p,q}x^py^q
	\end{equation}
	the \emph{principle symbol} of $T$ in \eqref{ele}. For example $\sigma(y^2x)=xy^2$ and $\sigma(y^2x^2)=x^2y^2$. The adjoint of $T$ is
	\begin{equation}
		T^*=\sum_{k=0}^n\sum_{p+q=k}\bar{a}_{p,q}y^qx^p,
	\end{equation}
	to transform $T^*$ into the standard form \eqref{ele}, we should represent $y^qx^p$ as a linear combination of $x^my^n$'s where $m,n\in\N$, like the above example $y^2x$ and $y^2x^2$. First, we have
	\begin{lemma}\label{n}
		For $n\in\N$, 
		\begin{equation}\label{n1}
			y^nx=nay^{n-1}+xy^n.
		\end{equation}
	\end{lemma}
	\begin{proof}
		The case $n=1$  follows directly from the definition.  If \eqref{n1} holds for some $n\in\N$, then \begin{align*}
			y^{n+1}x=y^n(xy+a)=ay^n+(nay^{n-1}+xy^n)y=(n+1)ay^n+xy^{n+1}
		\end{align*}
		and the conclusion follows by induction on $n$.
	\end{proof}

	For slightly more complex situations, when $n\geq 3$, Lemma \ref{n} immediately implies
	\begin{align*}
		y^nx^2&=(nay^{n-1}+xy^n)x\\
		&=na((n-1)ay^{n-2}+xy^{n-1})+x(nay^{n-1}+xy^n)\\
		&=n(n-1)a^2y^{n-2}+2naxy^{n-1}+x^2y^n,\\
		y^nx^3&=(n(n-1)a^2y^{n-2}+2naxy^{n-1}+x^2y^n)x\\
		&=n(n-1)a^2((n-2)ay^{n-3}+xy^{n-2})+2nax((n-1)ay^{n-2}+xy^{n-1})\\
		&\quad+x^2(nay^{n-1}+xy^n)\\
		&=n(n-1)(n-2)a^3y^{n-3}+3n(n-1)a^2xy^{n-2}+3nax^2y^{n-1}+x^3y^n,
	\end{align*}
	if $n<3$, say $n=1$, then 
	$$yx^3=(xy+a)x^2=ax^2+x(xy+a)x=2ax^2+x^2(xy+a)=3ax^2+x^3y,$$
	the second formula for $y^nx^3$ still holds. In fact, we have
	\begin{theorem}\label{xy}
		For all $m,n\in\N$,
		\begin{equation}\label{mn}
			y^nx^m=\sum_{p=0}^{m}\frac{n!}{(n-m+p)!}\binom{m}{p}a^{m-p}x^py^{n-m+p}.
		\end{equation}
		And hence $\sigma(y^nx^m)=x^my^n$.
	\end{theorem}
	\begin{proof}
		Lemma \ref{n} confirm the $m=1$ case. If \eqref{mn} holds for some $m$, then
		\begin{align*}
			y^nx^{m+1}&=\sum_{p=0}^{m}\frac{n!}{(n-m+p)!}\binom{m}{p}a^{m-p}x^py^{n-m+p}x\\
			&=\sum_{p=0}^{m}\frac{n!}{(n-m+p)!}\binom{m}{p}a^{m-p}x^p((n-m+p)ay^{n-m-1+p}+xy^{n-m+p})\\
			&=\sum_{p=0}^{m}\frac{n!}{(n-m+p-1)!}\binom{m}{p}a^{m-p+1}x^py^{n-m+p-1}+\sum_{p=0}^{m}\frac{n!}{(n-m+p)!}\binom{m}{p}a^{m-p}x^{p+1}y^{n-m+p}\\
			&=\sum_{p=1}^{m}\frac{n!}{(n-m+p-1)!}\left(\binom{m}{p}+\binom{m}{p-1}\right)a^{m-p+1}x^py^{n-m+p-1}+\frac{n!}{(n-m-1)!}a^{m+1}y^{n-m-1}\\
			&\quad+x^{m+1}y^n\\
			&=\sum_{p=1}^{m}\frac{n!}{(n-m+p-1)!}\binom{m+1}{p}a^{m-p+1}x^py^{n-m+p-1}+\frac{n!}{(n-m-1)!}a^{m+1}y^{n-m-1}+x^{m+1}y^n\\
			&=\sum_{p=0}^{m+1}\frac{n!}{(n-m-1+p)!}\binom{m+1}{p}a^{m+1-p}x^py^{n-m-1+p},
		\end{align*}
		i.e., \eqref{mn} also holds for $m+1$, so according to mathematical induction \eqref{mn} holds for all $m,n\in\N$.
	\end{proof} 
	For any projectors $p\in P(A(\mathbb{R}^2_\theta))$, we have $p^2=p=p^*$, then by applying Theorem \ref{xy}, compare the principle symbol of entries of $p^2$ and $p$, we can show that in fact $p\in P(\C)$, i.e., we have
	\begin{theorem}\label{pnr}
		If $p\in P(A(\mathbb{R}^2_\theta))$, then $p\in P(\C)$.
	\end{theorem}
	\begin{proof}
		Let $$p=\begin{pmatrix}
			p_{1,1}&\ldots&p_{1,n}\\
			\vdots&\ddots&\vdots\\
			p_{n,1}&\ldots&p_{n,n}\\
		\end{pmatrix}\in M_n(A(\mathbb{R}^2_\theta))$$
		be a $n$-dimensional projector of $A(\mathbb{R}^2_\theta)$. Then $p^2=p=p^*$, and
		\begin{align*}
			\begin{pmatrix}
				p_{1,1}&\ldots&p_{1,n}\\
				\vdots&\ddots&\vdots\\
				p_{n,1}&\ldots&p_{n,n}\\
			\end{pmatrix}=p=p^2&=pp^*\\
			&=\begin{pmatrix}
				p_{1,1}&\ldots&p_{1,n}\\
				\vdots&\ddots&\vdots\\
				p_{n,1}&\ldots&p_{n,n}\\
			\end{pmatrix}\cdot\begin{pmatrix}
				p_{1,1}^*&\ldots&p_{n,1}^*\\
				\vdots&\ddots&\vdots\\
				p_{1,n}^*&\ldots&p_{n,n}^*\\
			\end{pmatrix}\\
			&=\begin{pmatrix}
				\sum_{j=1}^np_{1,j}p_{1,j}^*&\ldots&\sum_{j=1}^np_{1,j}p_{n,j}^*\\
				\vdots&\ddots&\vdots\\
				\sum_{j=1}^np_{n,j}p_{1,j}^*&\ldots&\sum_{j=1}^np_{n,j}p_{n,j}^*\\
			\end{pmatrix},
		\end{align*}
		so for $1\leq k\leq n$,
		\begin{equation}\label{pll}
			p_{k,k}=\sum_{j=1}^np_{k,j}p_{k,j}^*.
		\end{equation}
		Let $\deg(p_{k,l})=N_{k,l}$ and
		\begin{equation}
			p_{k,l}=\sum_{j=0}^{N_{k,l}}\sum_{\alpha+\beta=j}a_{\alpha,\beta}^{k,l}x^\alpha y^\beta
		\end{equation}
		where all $a_{\alpha,\beta}^{k,l}\in\C$. Without loss of generality we assume $N_{k,1}\geq N_{k,l}$ for $1\leq l\leq n$. Then by applying Theorem \ref{xy}, 	
		\begin{align*}
			\sigma\left(\sum_{k=1}^np_{k,l}p_{k,l}^*\right)=&\sigma\left(\sum_{k=1}^n\left(\sum_{j=0}^{N_{k,l}}\sum_{\alpha+\beta=j}a_{\alpha,\beta}^{k,l}x^\alpha y^\beta\right)\left(\sum_{j=0}^{N_{k,l}}\sum_{\g+\d=j}\bar{a}_{\alpha,\beta}^{k,l}y^\beta x^\alpha\right)\right)\\
			=&\sigma\left(\sum_{k=1}^n\left(\sum_{\alpha+\beta=N_{k,l}}a_{\alpha,\beta}^{k,l}x^\alpha y^\beta\right)\left(\sum_{\g+\d=N_{k,l}}\bar{a}_{\alpha,\beta}^{k,l}y^\beta x^\alpha\right)\right)\\
			=&\sum_{k=1}^n\sum_{\alpha+\beta=N_{k,l}}\sum_{\g+\d=N_{k,l}}a_{\a,\b}^{k,l}\bar{a}_{\g,\d}^{k,l}x^{\a+\g}y^{\b+\d}\\
			=&\sum_{k=1}^n\sum_{j=0}^{2N_{k,l}}\sum_{\a=0}^{j}a_{\a,N_{k,l}-\a}^{k,l}\bar{a}_{j-\a,N_{k,l}-j+\a}^{k,l}x^{j}y^{2N_{k,l}-j}.
		\end{align*}
		If $N_{k,1}\geq1$ and $\sigma\left(\sum_{k=1}^np_{k,l}p_{k,l}^*\right)\neq 0$,  then
		\begin{equation}
			\sigma\left(\sum_{k=1}^np_{k,l}p_{k,l}^*\right)=2N_{k,1}>N_{k,1}=\sigma(p_{k,k}),
		\end{equation}
		contradicts to \eqref{pll}, so we must have
		\begin{equation}\label{slk}
			\sum_{k=1}^n\sum_{j=0}^{2N_{k,l}}\sum_{\a=0}^{j}a_{\a,N_{k,l}-\a}^{k,l}\bar{a}_{j-\a,N_{k,l}-j+\a}^{k,l}x^{j}y^{2N_{k,l}-j}=\sigma\left(\sum_{k=1}^np_{k,l}p_{k,l}^*\right)=0.
		\end{equation}
		For $1\leq k\leq n$, the coefficient of $y^{2N_{k,l}}$ is
		\begin{equation}
			0=\sum_{N_{k,j}=N_{k,l}}a_{0,N_{k,j}}^{k,j}\bar{a}_{0,N_{k,j}}^{k,j}=\sum_{N_{k,j}=N_{k,l}}\n a_{0,N_{k,j}}^{k,j}\n^2,
		\end{equation}
		so $a_{0,N_{k,l}}^{k,l}=0$ for all $1\leq k\leq n$; Then the coefficient of $x^2y^{2N_{k,l}-2}$ is
		\begin{equation}
			0=\sum_{N_{k,j}=N_{k,l}}\left(a_{0,N_{k,j}}^{k,j}\bar{a}_{2,N_{k,j}-2}^{k,j}+a_{1,N_{k,j}-1}^{k,j}\bar{a}_{1,N_{k,j}-1}^{k,j}+a_{2,N_{k,j}-2}^{k,j}\bar{a}_{0,N_{k,j}}^{k,j}\right)=\sum_{N_{k,j}=N_{k,l}}\n a_{1,N_{k,j}}^{k,j}\n^2,
		\end{equation}	
		so $a_{1,N_{k,l}-1}^{k,l}=0$ for all $1\leq k\leq n$; Next consider the coefficient of $x^4y^{2N_{k,l}-4}$, finally we have $a_{\a,N_{k,l}-\a}^{k,l}=0$ for all $1\leq k\leq n$ and $0\leq\a\leq N_{k,l}$, but this contradicts to the assumption $N_{k,1}\geq1$. So we must have $N_{k,1}=0$ and hence $p_{k,l}\in\C$ for all $1\leq k,l\leq n$, which implies $p\in P(\C)$.
	\end{proof}
	Then by the definition of $K_0$ groups, we have
	\begin{corollary}
		$K_0(A(\R^2_\theta))=K_0(\C)=\Z$.
	\end{corollary}
	For $A(\mathbb{R}^{2n}_\Theta)$, $n\geq2$, note that $x_3,x_4,\ldots,x_{2n}$ commute with $x_1$ and $x_2$, so for any $p\in P(A(\mathbb{R}^{2n}_\Theta))$, we temporarily regard $x_3,x_4,\ldots,x_{2n}$ as constants, then by applying the above method, we can also prove that $p\in P(\C)$. Hence we also have
	\begin{corollary}
		If $p\in P(A(\mathbb{R}^{2n}_\Theta))$, then $p\in P(\C)$ and hence $K_0(A(\R^{2n}_\Theta))=K_0(\C)=\Z$ for all $n\in\N$.
	\end{corollary}
	We have finished the proof of Theorem \ref{m1}.
	
	\section{The smooth case}\label{scp}
	
	For $A^\infty(\mathbb{R}^{2}_\theta)$, the characterization of projectors is much more difficult, even the projective elements of $A^\infty(\mathbb{R}^{2}_\theta)$.  Again, we start from the simplest case, $A^\infty(\mathbb{R}^{2}_\theta)$. Define 
	$$\g:=\frac{\sqrt{2\theta}(1+i)}{2},$$
	then $\g^2=i\theta=a$, $\g=i\bar{\g}$. Any element $T\in A^\infty(\mathbb{R}^{2}_\theta)$ can be represented as
	\begin{equation}\label{inft}
		T=\sum_{p,q=0}^\infty a_{p,q}\g^{-p-q}x^py^q\in A^\infty(\R_\theta^2)
	\end{equation}
	where $\{a_{p,q}\g^{-p-q}\}\in\mathcal{S}(\Z^2)$ (We will see the advantage of the representation \eqref{inft} later). We can no longer use the method in the proof of Theorem \ref{pnr} to study $P(A^\infty(\mathbb{R}^{2}_\theta))$ because we cannot define principle symbol for the infinite sum \eqref{inft}. First we have the following theorem, which provides the conditions that $\{a_{p,q}\}$ must satisfied when \eqref{inft} is a projector.
	
	\begin{theorem}\label{tt*}
		If \begin{equation}
			T=\sum_{p,q=0}^\infty a_{p,q}\g^{-p-q}x^py^q\in A^\infty(\R_\theta^2)
		\end{equation}
		satisfies $T=TT^*$, we have
		\begin{equation}\label{a*}
			\begin{aligned}
				a_{m,n}&=i^{m+n}\sum_{h=0}^\infty (-1)^h\binom{m+h}{m}\frac{(n+h)!}{n!}\bar{a}_{m+h,n+h}\\
				&=\sum_{r=0}^\infty\sum_{h=0}^{r}\sum_{s=0}^{n}\binom{r}{h}\frac{(n+h-s)!}{(n-s)!}a_{m+h-r,n+h-s}{a}_{r,s}.
			\end{aligned}
		\end{equation}
	\end{theorem}
	\begin{proof}
		If $T=TT^*$, then $T^*=(TT^*)^*=TT^*=T$, so $T\in P(A^\infty(\mathbb{R}^{2}_\theta))$.  And Theorem \ref{xy} implies
		\begin{align*}
			T=T^*=&\sum_{p,q=0}^\infty\bar{a}_{p,q}\bar{\g}^{-p-q}y^qx^p\\
			=&\sum_{p,q=0}^\infty\bar{a}_{p,q}i^{p+q}\g^{-p-q}\sum_{h=0}^{p}\frac{q!}{(q-p+h)!}\binom{p}{h}a^{p-h}x^hy^{q-p+h},\\
			T=T^2=&\sum_{p,q=0}^\infty a_{p,q}\g^{-p-q}x^py^q\sum_{r,s=0}^\infty{a}_{r,s}\g^{-r-s}x^ry^s\\
			=&\sum_{p,q=0}^\infty\sum_{r,s=0}^\infty a_{p,q}{a}_{r,s}\g^{-p-q-r-s}x^py^{q}x^ry^s\\
			=&\sum_{p,q=0}^\infty\sum_{r,s=0}^\infty a_{p,q}{a}_{r,s}\g^{-p-q-r-s}x^p\sum_{j=0}^{r}\frac{q!}{(q-r+j)!}\binom{r}{j}a^{r-j}x^jy^{q-r+j+s}\\
			=&\sum_{p,q=0}^\infty\sum_{r,s=0}^\infty\sum_{j=0}^{r}\g^{r-p-q-s-2j}\frac{q!}{(q-r+j)!}\binom{r}{j}a_{p,q}{a}_{r,s}x^{p+j}y^{q+s-r+j}.
		\end{align*}

		For $m,n\in\N$, compare the coefficients of $x^my^n$ on both sides, in the first part of $TT^*$ the indices $j,p,q,r,s$ should satisfy the following constrains:
		\begin{equation}
			\begin{cases}
				p+j=m\\
				q+s-r+j=n,
			\end{cases}
		\end{equation}
		then
		\begin{equation}
			\begin{cases}
				p=m-j\\
				q=n+r-j-s,
			\end{cases}
		\end{equation}
		so we have
		\begin{align*}
			a_{m,n}\g^{-m-n}&=\sum_{h=0}^\infty i^{m+n+2h}\binom{m+h}{m}\frac{(n+h)!}{n!}\bar{a}_{m+h,n+h}{\g}^{-m-n-2h}a^{h}\\
			&=\sum_{h=0}^\infty i^{m+n+2h}\binom{m+h}{m}\frac{(n+h)!}{n!}\bar{a}_{m+h,n+h}{\g}^{-m-n},\\
			a_{m,n}\g^{-m-n}&=\sum_{r,s=0}^\infty\sum_{j=0}^{r-1}\frac{(n+r-j-s)!}{(n-s)!}\binom{r}{j}a_{m-j,n+r-j-s}{a}_{r,s}\g^{-m-n}\\
			&=\sum_{r=0}^\infty\sum_{j=0}^{r}\sum_{s=0}^{n}\binom{r}{j}\frac{(n+r-j-s)!}{(n-s)!}a_{m-j,n+r-j-s}{a}_{r,s}\g^{-m-n}\\
			&=\sum_{r=0}^\infty\sum_{h=0}^{r}\sum_{s=0}^{n}\binom{r}{h}\frac{(n+h-s)!}{(n-s)!}a_{m+h-r,n+h-s}{a}_{r,s}\g^{-m-n},
		\end{align*}
		and hence
		\begin{align*}
			a_{m,n}&=i^{m+n}\sum_{h=0}^\infty (-1)^h\binom{m+h}{m}\frac{(n+h)!}{n!}\bar{a}_{m+h,n+h}\\
			&=\sum_{r=0}^\infty\sum_{h=0}^{r}\sum_{s=0}^{n}\binom{r}{h}\frac{(n+h-s)!}{(n-s)!}a_{m+h-r,n+h-s}{a}_{r,s}.
		\end{align*}
	\end{proof}

	\subsection{$P_k(A^\infty(\R_\theta^2))$ where $k\geq0$.}
	It's too general to find all series $\{a_{m,n}\}$ satisfying \eqref{a*}. Let's first consider a simple case, $a_{m,n}=0$ where $m\neq n$. We will need the following well-known binomial formula:
	\begin{equation}\label{wkb}
		\sum_{k=0}^m\binom{m}{k}\binom{n}{l-k}=\binom{m+n}{l}.
	\end{equation}
	It can be proved by comparing the coefficient of $z^l$ in the polynomial $(1+z)^m(1+z)^n=(1+z)^{m+n}$. We have
	\begin{proposition}\label{p0}
		If \begin{equation}
			S:=\sum_{p=0}^\infty a_{p,p}\g^{-2p}x^py^p\in P(A^\infty(\R_\theta^2)),
		\end{equation}
		then $S=0$ or $1$.
	\end{proposition}
	\begin{proof}
		Define $b_p:=a_{p,p}$, then from \eqref{a*} we have
		\begin{equation}\label{prob}
			b_m=\sum_{r=0}^m\sum_{h=0}^r\binom{r}{h}\frac{(m+h-r)!}{(m-r)!}b_{m+h-r}b_r.
		\end{equation}
		Let $\mfb_m:=m!b_m$, then
		\begin{equation}\label{promb}
			\mfb_m=\sum_{r=0}^m\sum_{h=0}^r\frac{m!}{h!(r-h)!(m-r)!}\mfb_{m+h-r}\mfb_r
		\end{equation}
		and for $m\in\N$,
		\begin{align*}
			&\sum_{p=0}^m\binom{m}{p}\mfb_p\\
			=&\sum_{p=0}^m\binom{m}{p}\sum_{r=0}^p\sum_{h=0}^r\frac{p!}{h!(r-h)!(p-r)!}\mfb_{p+h-r}\mfb_r\\
			=&\sum_{p=0}^m\sum_{r=0}^p\sum_{h=0}^r\frac{m!}{(m-p)!(p-r)!(r-h)!h!}\mfb_{p+h-r}\mfb_r\\
			=&\sum_{p=0}^m\sum_{q=0}^m\sum_{h=0}^m\frac{m!}{(m+h-p-q)!(p-h)!(q-h)!h!}\mfb_p\mfb_q\\
			=&\sum_{p=0}^m\sum_{q=0}^m\binom{m}{p}\sum_{h=0}^m\binom{m-p}{q-h}\binom{p}{h}\mfb_p\mfb_q\\
			=&\sum_{p=0}^m\sum_{q=0}^m\binom{m}{p}\binom{m}{q}\mfb_p\mfb_q\\
			=&\left(\sum_{p=0}^m\binom{m}{p}\mfb_p\right)^2,
		\end{align*}
		hence
		\begin{equation}\label{b01}
			\sum_{p=0}^m\binom{m}{p}\mfb_p=0\text{ or }1.
		\end{equation}
		When $m\geq1$,
		\begin{equation}\label{bm01}
			\mfb_m=-\sum_{p=0}^{m-1}\binom{m}{p}\mfb_p\text{ or }1-\sum_{p=0}^{m-1}\binom{m}{p}\mfb_p,
		\end{equation}
		for $m=0$, \eqref{b01} implies $\mfb_0=0$ or 1, then conbine with \eqref{bm01} we can see that $\mfb_m\in\N$ for every $m\in\N$. Also, \eqref{a*} implies
		\begin{equation}
			b_{m}=(-1)^m\sum_{h=0}^\infty (-1)^h\binom{m+h}{m}\frac{(m+h)!}{m!}\bar{b}_{m+h},
		\end{equation}
		so
		\begin{equation}\label{bm}
			\mfb_{m}=(-1)^m\sum_{h=0}^\infty (-1)^h\binom{m+h}{m}\bar{\mfb}_{m+h}.
		\end{equation}
		Note that 
		\begin{equation}
			\lim_{h\to\infty}\binom{m+h}{m}=1\text{ or }+\infty,
		\end{equation}
		so the  sum \eqref{bm} should be finite, otherwise it won't convergent. So in fact $S\in P(A(\R_\theta^2))$, and then Theorem \ref{pnr} implies $S=0$ or 1.
	\end{proof}
	On the basis of Proposition \ref{p0} we can consider more general cases. 
	\begin{definition}\label{pk2}
		For $k\in\N$, we define
		$$P_k(A^\infty(\R_\theta^2)):=\{\sum_{p,q=0}^\infty a_{p,q}\g^{-p-q}x^py^q\in P(A^\infty(\R_\theta^2)):a_{p,q}=0\text{ for }\n p-q\n>k.\}.$$
	\end{definition}
	Obviously we have $P_k(A^\infty(\R_\theta^2))\subseteq P_l(A^\infty(\R_\theta^2))$ for $k\leq l$. Proposition \ref{p0} shows that $P_0(A^\infty(\R_\theta^2))=\{0,1\}$. With the help of Conjecture \ref{cjj}, we can prove that 
	\begin{theorem}\label{pk}
		$P_k(A^\infty(\R_\theta^2))=\{0,1\}$ for all $k\in\N$.
	\end{theorem}
	\begin{proof}The case of $k=0$ can be derived from Proposition \ref{p0}.  If $P_{k-1}(A^\infty(\R_\theta^2))=\{0,1\}$ for a positive $k$, let
		$$\mathcal{P}:=\sum_{p,q=0}^\infty a_{p,q}\g^{-p-q}x^py^q\in P_k(A^\infty(\R_\theta^2)),$$
		then Theorem \ref{tt*} and the definition of $P_k(A^\infty(\R_\theta^2))$ imply that 
		\begin{align*}
			0=a_{m,m+2k}=&\sum_{r=0}^\infty\sum_{h=0}^{r}\sum_{s=0}^{m+2k}\binom{r}{h}\frac{(m+2k+h-s)!}{(m+2k-s)!}a_{m+h-r,m+2k+h-s}{a}_{r,s}\\
			=&\sum_{r=0}^{m+k}\sum_{h=0}^{r}\binom{r}{h}\frac{(m-r+k+h)!}{(m-r+k)!}a_{m+h-r,m+h-r+k}{a}_{r,r+k},
		\end{align*}
		hence for $n\in\N$,
		\begin{align*}
			0&=\sum_{p=0}^{n}\frac{n!}{(n-p)!}a_{p,p+2k}\\
			&=\sum_{p=0}^{n}\frac{n!}{(n-p)!}\sum_{r=0}^{p+k}\sum_{h=0}^{r}\binom{r}{h}\frac{(p-r+k+h)!}{(p-r+k)!}a_{p+h-r,p+h-r+k}{a}_{r,r+k}\\
			&=\sum_{q=0}^{n}a_{q,q+k}\sum_{p=q}^{n}\sum_{r=0}^{p+k}\frac{n!}{(n-p)!}\binom{r}{q+r-p}\frac{(q+k)!}{(p+k-r)!}a_{r,r+k}\\
			&=\sum_{q=0}^{n}a_{q,q+k}\sum_{r=0}^{n+k}\left(\sum_{p=q}^{q+r}\frac{n!}{(n-p)!}\binom{r}{q+r-p}\frac{(q+k)!}{(p+k-r)!}\right)a_{r,r+k}\\
			&=\sum_{q=0}^{n}\frac{n!}{(n-q)!}a_{q,q+k}\sum_{r=0}^{n+k}\left(\sum_{p=0}^{r}\frac{(n-q)!}{(n-p-q)!}\binom{r}{r-p}\frac{(q+k)!}{(p+q+k-r)!}\right)a_{r,r+k}\\
			&=\sum_{q=0}^{n}\frac{n!}{(n-q)!}a_{q,q+k}\sum_{r=0}^{n+k}r!\sum_{p=0}^{r}\binom{n-q}{p}\binom{q+k}{r-p}a_{r,r+k}\\
			&=\sum_{q=0}^{n}\frac{n!}{(n-q)!}a_{q,q+k}\sum_{r=0}^{n+k}\frac{(n+k)!}{(n+k-r)!}a_{r,r+k}.\quad\text{Here we apply \eqref{wkb}.}
		\end{align*}
		Let $\mfa_q:=q!a_{q,q+k}$ and
		\begin{equation}
			b_n:=\sum_{q=0}^n\frac{n!}{(n-q)!}a_{q,q+k}=\sum_{q=0}^n\binom{n}{q}\mfa_{q},
		\end{equation}
		then $b_{n}b_{n+k}=0$, and the binomial transform formula tells us that
		\begin{equation}\label{mfab}
			\mfa_{n}=\sum_{q=0}^n(-1)^{n-q}\binom{n}{q}b_q.
		\end{equation}
		Also for $n\in\N$, Theorem \ref{tt*} implies
		\begin{equation}
			a_{n,n+k}=i^{2n+k}\sum_{h=0}^\infty (-1)^h\binom{n+h}{n}\frac{(n+k+h)!}{(n+k)!}\bar{a}_{n+h,n+k+h},
		\end{equation}
		hence
		\begin{equation}
			\mfa_0=a_{0,k}=i^k\sum_{h=0}^{\infty}(-1)^{h}\frac{(h+k)!}{k!}\bar{a}_{h,h+k}=i^k\sum_{h=0}^{\infty}(-1)^{h}\binom{h+k}{k}\bar{\mfa}_{h},
		\end{equation}
		which means 
		\begin{equation}
			\lim_{h\to\infty}\binom{h+k}{k}\bar{\mfa}_{h}=0,
		\end{equation}
		and so we must have $\lim_{h\to\infty}{\mfa}_{h}=0$. Then if Conjecture \ref{cjj} is true, we have $b_n\equiv0$ and \eqref{mfab} implies $\mfa_n\equiv0$, $a_{n,n+k}\equiv0$. Similarly we can also prove that $a_{n+k,n}\equiv0$. Hence in fact $\mathcal{P}\in P_{k-1}(A^\infty(\R_\theta^2))$, so $P_{k}(A^\infty(\R_\theta^2))=\{0,1\}$ and by induction on $k$, $P_{k}(A^\infty(\R_\theta^2))=\{0,1\}$ for any $k\in\N$.
	\end{proof}
	What we have done is just a beginning, we left the study of more general cases, or all solutions of \eqref{a*} for future. If that was finished, we may be able to characterize $P(A^\infty(\R_\theta^2))$, which will tell us what $K_0(A^\infty(\R_\theta^2))$ is.
	
	\subsection{$P_0(A^\infty(\R_\Theta^{2n}))$ where $n\geq2$.}We extend Definition \ref{pk2} to higher-dimensional case.
	\begin{definition}\label{pkn}
		For $k,n\in\N$, we  define
		\begin{align*}
			P_k(A^\infty(\R_\Theta^{2n})):=&\{\sum_{p_1,\ldots,p_{2n}=0}^\infty a_{p_1,\ldots,p_{2n}}\g_1^{-p_1-p_2}\ldots\g_n^{-p_{2n-1}-p_{2n}}x_1^{p_1}\ldots x_{2n}^{p_{2n}}\in P(A^\infty(\R_\Theta^{2n})):a_{p_1,\ldots,p_{2n}}=0\\
			&\text{ where }\max_{1\leq r,s\leq 2n}\n p_r-p_s\n>k.\}
		\end{align*}
		where 
		$$\g_m:=\frac{\sqrt{2\theta_{2m-1,2m}}(1+i)}{2},~m=1,2,\ldots,n.$$
	\end{definition}

	Now, let's consider a slightly more complicated example, $P_0(A^\infty(\R_\Theta^{4}))$. If
	$$\mathcal{P}:=\sum_{p=0}^\infty b_{p}\g_1^{-2p}\g_2^{-2p}x_1^px_2^px_3^px_4^p:=\sum_{p=0}^\infty a_{p,p,p,p}\g_1^{-2p}\g_2^{-2p}x_1^px_2^px_3^px_4^p\in P_0(A^\infty(\R_\Theta^{4})),$$
	from \eqref{prob} we have
	\begin{align*}
		b_m\g_2^{-2m}x_3^mx_4^m=&\sum_{r=0}^m\sum_{h=0}^r\binom{r}{h}\frac{(m+h-r)!}{(m-r)!}b_{m+h-r}\g_2^{-2(m+h-r)}b_r\g_2^{-2r}x_3^{m+h-r}x_4^{m+h-r}x_3^rx_4^r\\
		=&\sum_{r=0}^m\sum_{h=0}^r\binom{r}{h}\frac{(m+h-r)!}{(m-r)!}b_{m+h-r}b_r\g_2^{-2(m+h)}\sum_{j=0}^r\binom{r}{j}\frac{(m+h-r)!}{(m+h-2r+j)!}\\
		&\times\g_2^{2r-2j}x_3^{m+h-r+j}x_4^{m+h-r+j},
	\end{align*}
	then
	\begin{equation}\label{prob2}
		b_m=\sum_{r=0}^m\sum_{h=0}^r\left(\binom{r}{h}\frac{(m+h-r)!}{(m-r)!}\right)^2b_{m+h-r}b_r.
	\end{equation}
	This time we can no long conclude that $\{m!b_m\}_{m\geq 0}$ are all integers, for instance,
	\begin{equation}
		\{b_0,b_1,b_2,b_3,\ldots\}=\{0,1,-\frac{1}{8} \left(7+\sqrt{33}\right),\frac{1}{72} \left(46+9
		\sqrt{33}-\sqrt{2089+360 \sqrt{33}}\right),\ldots\}
	\end{equation}
	is a set of solution of \eqref{prob2}. $\mathcal{P}=\mathcal{P}^*$ implies
	\begin{align*}
		b_{m}\g_2^{-2m}x_3^mx_4^m&=(-1)^m\sum_{h=0}^\infty (-1)^h\binom{m+h}{m}\frac{(m+h)!}{m!}\overline{b_{m+h}\g_2^{-2(m+h)}}x_4^{m+h}x_3^{m+h}\\
		&=\sum_{h=0}^\infty \binom{m+h}{m}\frac{(m+h)!}{m!}\bar{b}_{m+h}\g_2^{-2(m+h)}\sum_{j=0}^{m+h}\binom{m+h}{j}\frac{(m+h)!}{j!}\g_2^{2(m+h-j)}x_3^jx_4^j,
	\end{align*}
	so
	\begin{equation}
		b_m=\sum_{h=0}^\infty\left(\binom{m+h}{m}\frac{(m+h)!}{m!}\right)^2\bar{b}_{m+h}.
	\end{equation}
	Like before, if we define $\mfb_{m,2}:=(m!)^2b_{m}$, then
	\begin{equation}\label{mfbm2}
		\mfb_{m,2}=\sum_{r=0}^m\sum_{h=0}^r\left(\frac{m!}{(m-r)!(r-h)!h!}\right)^2\mfb_{m+h-r,2}\mfb_{r,2}=\sum_{h=0}^\infty\binom{m+h}{m}^2\bar{\mfb}_{m+h,2},~\forall m\in\N.
	\end{equation}
	Similarly, every element of $P_0(A^\infty(\R_\Theta^{2n}))$ corresponds to a sequence $\{\mfb_{m,n}\}_{m\geq0}$ satisfying
	\begin{equation}\label{p0n}
		\mfb_{m,n}=\sum_{r=0}^m\sum_{h=0}^r\left(\frac{m!}{(m-r)!(r-h)!h!}\right)^n\mfb_{m+h-r,n}\mfb_{r,n}=\sum_{h=0}^\infty(-1)^{(m+h)n}\binom{m+h}{m}^n\bar{\mfb}_{m+h,n},~\forall m\in\N.
	\end{equation}
	It's quite possible that $P_0(A^\infty(\R_\Theta^{2n}))=\{0,1\}$ for $\forall n\in\N$, so we propose the following conjecture.
	\begin{conjecture}
		Let $n\in\N^*$ and  if a sequence $\{\mfb_{m,n}\}_{m\geq0}$ satisfies \eqref{p0n}, then
		$$\mfb_{m,n}=\begin{cases}
			0\text{ or }1,&m=0,\\
			0,&m\geq1,
		\end{cases}$$
		which means $P_0(A^\infty(\R_\Theta^{2n}))=\{0,1\}$ for $\forall n\in\N^*$.
	\end{conjecture}
	
	\subsection{$P_k(A^\infty(\R_\Theta^{4}))$ where $k\geq1$.}
	For the $k\geq1$ case, as a warm-up, we first consider $P_1(A^\infty(\R_\Theta^{4}))$. If
	$$\mathcal{P}:=\sum_{p_1,p_2,p_3,p_4=0}^\infty a_{p_1,p_2,p_3,p_4}\g_1^{-p_1-p_2}\g_2^{-p_3-p_4}x_1^{p_1}x_2^{p_2}x_3^{p_3}x_4^{p_4}\in P_1(A^\infty(\R_\Theta^{4})),$$
	rewrite $\mathcal{P}$ to
	\begin{equation}
		\sum_{p_1,p_2=0}^\infty\left(\sum_{p_3,p_4=0}^\infty a_{p_1,p_2,p_3,p_4}\g_2^{-p_3-p_4}x_3^{p_3}x_4^{p_4}\right)\g_1^{-p_1-p_2}x_1^{p_1}x_2^{p_2}
	\end{equation}
	and set
	\begin{equation}
		\a_{p_1,p_2}:=\sum_{p_3,p_4=0}^\infty a_{p_1,p_2,p_3,p_4}\g_2^{-p_3-p_4}x_3^{p_3}x_4^{p_4},~\b_{n}:=\sum_{m=0}^n\frac{n!}{(n-p)!}\a_{p,p+1},
	\end{equation}
	then from the proof of Theorem \ref{pk} we have $\b_{n}\b_{n+1}=0$. Note that
	\begin{equation}
		\a_{m,m+1}=\sum_{p_3,p_4=0}^\infty a_{m,m+1,p_3,p_4}\g_2^{-p_3-p_4}x_3^{p_3}x_4^{p_4}=\sum_{r=0}^1\sum_{s=0}^1 a_{m,m+1,m+r,m+s}\g_2^{-2m-r-s}x_3^{m+r}x_4^{m+s},
	\end{equation}
	so
	\begin{align*}
		0&=\b_{n}\b_{n+1}\\
		&=\sum_{p=0}^n\frac{n!}{(n-p)!}\sum_{r=0}^1\sum_{s=0}^1 a_{p,p+1,p+r,p+s}\g_2^{-2p-r-s}x_3^{p+r}x_4^{p+s}\sum_{q=0}^{n+1}\frac{(n+1)!}{(n+1-q)!}\sum_{r=0}^1\sum_{s=0}^1a_{q,q+1,q+r,q+s}\\
		&\quad\times\g_2^{-2q-r-s}x_3^{q+r}x_4^{q+s}\\
		&=\sum_{p=0}^n\sum_{q=0}^{n+1}\sum_{r=0}^1\sum_{s=0}^1\sum_{u=0}^1\sum_{v=0}^1\frac{n!(n+1)!}{(n-p)!(n+1-q)!} a_{p,p+1,p+r,p+s}a_{q,q+1,q+u,q+v}\g_2^{-2p-2q-r-s-u-v}\\
		&\quad\times x_3^{p+r}x_4^{p+s}x_3^{q+u}x_4^{q+v}\\
		&=\sum_{p=0}^n\sum_{q=0}^{n+1}\sum_{r=0}^1\sum_{s=0}^1\sum_{u=0}^1\sum_{v=0}^1\frac{n!(n+1)!}{(n-p)!(n+1-q)!} a_{p,p+1,p+r,p+s}a_{q,q+1,q+u,q+v}\g_2^{-2p-2q-r-s-u-v}\\
		&\quad\times\sum_{h=0}^{q+u}\binom{q+u}{h}\frac{(p+s)!}{(p+s-q-u+h)!}\g_2^{2q+2u-2h}x_3^{p+r+h}x_4^{p+s-u+v+h},
	\end{align*}
	consider the coefficient of $x_3^mx_4^{m+2}$ where $m\in\N$, we have
	\begin{equation}\label{mm2}
		\sum_{p=0}^m\sum_{q=m-p}^{m+1}\frac{n!(n+1)!}{(n-p)!(n+1-q)!}\binom{q}{m-p}\frac{(p+1)!}{(m+1-q)!}a_{p,p+1,p,p+1}a_{q,q+1,q,q+1}=0.
	\end{equation}
	Let $p=q=m/2$, \eqref{mm2k} implies $a_{p,p+1,p,p+1}^2=0$, so $a_{p,p+1,p,p+1}=0$, $\forall p\in\N$.  Similarity, by considering the coefficients of $x_3^{m+2}x_4^{m}$ we have $a_{p,p+1,p+1,p}=0$ for $\forall p\in\N$, and then
	\begin{align*}
		0=&\b_{n}\b_{n+1}\\
		=&\sum_{p=0}^n\sum_{q=0}^{n+1}\sum_{r=0}^1\sum_{u=0}^1\frac{n!(n+1)!}{(n-p)!(n+1-q)!} a_{p,p+1,p+r,p+r}a_{q,q+1,q+u,q+u}\sum_{h=0}^{q+u}\binom{q+u}{h}\frac{(p+r)!}{(p+r-q-u+h)!}\\
		&\times\g_2^{-2p-2r-2h}x_3^{p+r+h}x_4^{p+r+h}.
	\end{align*}
	Then consider the coefficients of $x_3^mx_4^{m}$, we have
	\begin{equation}\label{mm0111}
		\begin{aligned}
			&\sum_{p=0}^m\sum_{q=0}^{m+k}\frac{n!(n+1)!}{(n-p)!(n+1-q)!}\sum_{r=0}^1\sum_{u=0}^1 \binom{q+u}{m-p-r}\frac{(p+r)!}{(m-q-u)!}\\
			\times&a_{p,p+1,p+r,p+r}a_{q,q+1,q+u,q+u}=0,~\forall n\in\N,
		\end{aligned}
	\end{equation}
	Let $p=q=m/2-1$, then $a_{p,p+1,p+1,p+1}^2=0$, so $a_{p,p+1,p+1,p+1}=0$, $\forall p\in\N$. And next \eqref{mm0111} implies $a_{p,p+1,p,p}=0$ for $\forall p\in\N$.  Similarly we can prove $a_{p+1,p,p,p}=a_{p+1,p,p+1,p}=a_{p+1,p,p,p+1}=a_{p+1,p,p+1,p+1}=0$ for $\forall p\in\N$. Then from \eqref{prob} we have
	\begin{equation}\label{amm}
		\a_{m,m}=\sum_{r=0}^m\sum_{h=0}^r\binom{r}{h}\frac{(m+h-r)!}{(m-r)!}\a_{m+h-r,m+h-r}\a_{r,r}=\sum_{h=0}^\infty\binom{m+h}{m}\frac{(m+h)!}{m!}\a_{m+h,m+h}^*.
	\end{equation}
	Consider the coefficient of $x_3^{2m+2}x_4^{2m+2}$, we have $0=m!a_{m,m,m+1,m+1}^2,$ so $a_{m,m,m+1,m+1}=0$, then consider the coefficients of $x_3^{2m}x_4^{2m+2}$ and $x_3^{2m+2}x_4^{2m}$ we have $a_{m,m,m,m+1}=a_{m,m,m+1,m}=0,$ set $m=0$ we have $a_{0,0,0,0}=a_{0,0,0,0}^2,$ so $a_{0,0,0,0}=0$ or 1. 
	
	For $m\geq1$, consider the coefficient of $x_3^{2m}x_4^{2m}$, we can get $a_{m,m,m,m}=0$, then consider the coefficient of $x_3^{2m-2}x_4^{2m}$ and $x_3^{2m}x_4^{2m-2}$, we get that $a_{m,m,m-1,m}=a_{m,m,m,m-1}=0$. Now \eqref{amm} implies
	\begin{align*}
		&a_{m,m,m-1,m-1}\g_2^{-2m+2}x_3^{m-1}x_4^{m-1}\\
		=&\sum_{r=1}^{m-1}\sum_{h=0}^r\binom{r}{h}\frac{(m+h-r)!}{(m-r)!}a_{m+h-r,m+h-r,m+h-r-1,m+h-r-1} a_{r,r,r-1,r-1}\g_2^{-2m-2h+4}\\
		&\times x_3^{m+h-r-1}x_4^{m+h-r-1}x_3^{r-1}x_4^{r-1}+2a_{0,0,0,0}a_{m,m,m-1,m-1}\g_2^{-2m+2}x_3^{m-1}x_4^{m-1}+\sum_{h=1}^m\binom{m}{h}h!\\
		&\times a_{h,h,h-1,h-1}a_{m,m,m-1,m-1}\g_2^{-2m-2h+4}x_3^{h-1}x_4^{h-1}x_3^{m-1}x_4^{m-1}\\
		=&\sum_{r=1}^{m-1}\sum_{h=0}^r\binom{r}{h}\frac{(m+h-r)!}{(m-r)!}a_{m+h-r,m+h-r,m+h-r-1,m+h-r-1}a_{r,r,r-1,r-1}\g_2^{-2m-2h+4}\sum_{j=0}^{r-1}\binom{r-1}{j}\\
		&\times\frac{(m+h-r-1)!}{(m+h-2r+j)!}\g_2^{2r-2j-2}x_3^{m+h-r+j-1}x_4^{m+h-r+j-1}+2a_{0,0,0,0}a_{m,m,m-1,m-1}\g_2^{-2m+2}x_3^{m-1}x_4^{m-1}\\
		&+\sum_{h=1}^m\binom{m}{h}h!a_{h,h,h-1,h-1}a_{m,m,m-1,m-1}\sum_{j=0}^{m-1}\binom{m-1}{j}\frac{(h-1)!}{(h-m+j)!}\g_2^{-2h-2j+2}x_3^{h+j-1}x_4^{h+j-1},
	\end{align*}
	denote $\mfa_{m,2}:=a_{m,m,m-1,m-1}$ for $m\geq1$ and $\mfa_{0,2}:=a_{0,0,0,0}$, then
	\begin{equation}\label{mfam2}
		\begin{aligned}
			\mfa_{m,2}=&\sum_{r=1}^{m-1}\sum_{h=1}^r\binom{r}{h}\frac{(m+h-r)!}{(m-r)!}\binom{r-1}{h-1}\frac{(m+h-r-1)!}{(m-r)!}\mfa_{m+h-r,2}\mfa_{r,2}+2\mfa_{0,2}\mfa_{m,2}\\
			&+\sum_{h=1}^m\frac{m!(m-1)!}{(m-h)!^2}\mfa_{h,2}\mfa_{m,2},~\forall m\geq1.
		\end{aligned}
	\end{equation}
	For $m\geq1$, \eqref{amm} also implies
	\begin{align*}
		\sum_{r,s=0}^\infty a_{m,m,m-1,m-1}\g_2^{-2m-2}x_3^{m-1}x_4^{m-1}=&i^{m+m}\sum_{h=0}^{\infty}(-1)^h\binom{m+h}{m}\frac{(m+h)!}{m!}\sum_{r,s=0}^\infty\bar{a}_{m+h,m+h,r,s}i^{r+s}\\
		&\times\g_2^{-r-s}x_4^{s}x_3^{r}\\
		=&i^{m+m}\sum_{h=0}^{\infty}(-1)^h\binom{m+h}{m}\frac{(m+h)!}{m!}\sum_{r,s=0}^\infty \bar{a}_{m+h,m+h,r,s}i^{r+s}\\
		&\times\g_2^{-r-s}\sum_{j=0}^{r}\binom{r}{j}\frac{s!}{(s-r+j)!}\g_2^{2r-2j}x_3^jx_4^{s-r+j},
	\end{align*}
	hence
	\begin{equation}\label{mfam20}
		\mfa_{m,2}=\left\{\begin{aligned}
			&-\sum_{h=0}^{\infty}\binom{m+h}{m}\frac{(m+h)!}{m!}\binom{m+h-1}{m-1}\frac{(m+h-1)!}{(m-1)!}\bar{\mfa}_{m+h,2},&m\geq1,\\
			&\sum_{h=0}^\infty h!^2\bar{\mfa}_{h,2},&m=0.
		\end{aligned}
		\right.
	\end{equation}
	A sequence $\{\mfa_{m,2}\}_{m\geq0}$ satisfying \eqref{mfam2} is
	$$0,0,0,\frac{1}{12},-\frac{1}{288}
	\left(23+\sqrt{241}\right),\frac{341+20 \sqrt{241}-\sqrt{60281+4040
			\sqrt{241}}}{5760},\ldots$$
	but currently I don't find a sequence satisfying \eqref{mfam2} and \eqref{mfam20} simultaneously except zero sequence. So we propose the following conjecture.
	\begin{conjecture}\label{mfam}
		There is only zero sequence which satisfies \eqref{mfam2} and \eqref{mfam20} simultaneously.
	\end{conjecture}
	If Conjecture \ref{mfam} is true, then $\mathcal{P}\in P_0(A^\infty(\R_\Theta^{4}))$, which means $P_1(A^\infty(\R_\Theta^{4}))=P_0(A^\infty(\R_\Theta^{4}))$. For $P_k(A^\infty(\R_\Theta^{4}))$ where $k\geq2$, we can do similarly. If
	$$\mathcal{P}:=\sum_{p_1,p_2,p_3,p_4=0}^\infty a_{p_1,p_2,p_3,p_4}\g_1^{-p_1-p_2}\g_2^{-p_3-p_4}x_1^{p_1}x_2^{p_2}x_3^{p_3}x_4^{p_4}\in P_k(A^\infty(\R_\Theta^{4})),$$
	again we rewrite $\mathcal{P}$ to
	\begin{equation}
		\sum_{p_1,p_2=0}^\infty\left(\sum_{p_3,p_4=0}^\infty a_{p_1,p_2,p_3,p_4}\g_2^{-p_3-p_4}x_3^{p_3}x_4^{p_4}\right)\g_1^{-p_1-p_2}x_1^{p_1}x_2^{p_2}
	\end{equation}
	and set
	\begin{equation}
		\a_{p_1,p_2}:=\sum_{p_3,p_4=0}^\infty a_{p_1,p_2,p_3,p_4}\g_2^{-p_3-p_4}x_3^{p_3}x_4^{p_4},~\b_{n}:=\sum_{m=0}^n\frac{n!}{(n-p)!}\a_{p,p+k},
	\end{equation}
	then from the proof of Theorem \ref{pk} we have $\b_{n}\b_{n+k}=0$. Note that
	\begin{equation}
		\a_{m,m+k}=\sum_{p_3,p_4=0}^\infty a_{m,m+k,p_3,p_4}\g_2^{-p_3-p_4}x_3^{p_3}x_4^{p_4}=\sum_{r=0}^k\sum_{s=0}^k a_{m,m+k,m+r,m+s}\g_2^{-2m-r-s}x_3^{m+r}x_4^{m+s},
	\end{equation}
	so
	\begin{align*}
		0&=\b_{n}\b_{n+k}\\
		&=\sum_{p=0}^n\frac{n!}{(n-p)!}\sum_{r=0}^k\sum_{s=0}^k a_{p,p+k,p+r,p+s}\g_2^{-2p-r-s}x_3^{p+r}x_4^{p+s}\sum_{q=0}^{n+k}\frac{(n+k)!}{(n+k-q)!}\sum_{r=0}^k\sum_{s=0}^ka_{q,q+k,q+r,q+s}\\
		&\quad\times\g_2^{-2q-r-s}x_3^{q+r}x_4^{q+s}\\
		&=\sum_{p=0}^n\sum_{q=0}^{n+k}\sum_{r=0}^k\sum_{s=0}^k\sum_{u=0}^k\sum_{v=0}^k\frac{n!(n+k)!}{(n-p)!(n+k-q)!} a_{p,p+k,p+r,p+s}a_{q,q+k,q+u,q+v}\g_2^{-2p-2q-r-s-u-v}\\
		&\quad\times x_3^{p+r}x_4^{p+s}x_3^{q+u}x_4^{q+v}\\
		&=\sum_{p=0}^n\sum_{q=0}^{n+k}\sum_{r=0}^k\sum_{s=0}^k\sum_{u=0}^k\sum_{v=0}^k\frac{n!(n+k)!}{(n-p)!(n+k-q)!} a_{p,p+k,p+r,p+s}a_{q,q+k,q+u,q+v}\g_2^{-2p-2q-r-s-u-v}\\
		&\quad\times\sum_{h=0}^{q+u}\binom{q+u}{h}\frac{(p+s)!}{(p+s-q-u+h)!}\g_2^{2q+2u-2h}x_3^{p+r+h}x_4^{p+s-u+v+h},
	\end{align*}
	consider the coefficient of $x_3^mx_4^{m+2k}$ where $m\in\N$, we have
	\begin{equation}\label{mm2k}
		\sum_{p=0}^m\sum_{q=m-p}^{m+k}\frac{n!(n+k)!}{(n-p)!(n+k-q)!}\binom{q}{m-p}\frac{(p+k)!}{(m+k-q)!}a_{p,p+k,p,p+k}a_{q,q+k,q,q+k}=0, \forall n\in\N.
	\end{equation}
	Let $p=q=m/2$, \eqref{mm2k} implies $a_{p,p+k,p,p+k}^2=0$, so $a_{p,p+k,p,p+k}=0$, $\forall p\in\N$. Then consider the coefficients of $x_3^mx_4^{m+2k-2}$, we have
	\begin{equation}
		\begin{aligned}
			&\sum_{p=0}^m\sum_{q=0}^{m+k}\frac{n!(n+k)!}{(n-p)!(n+k-q)!}\sum_{r=0}^1\sum_{u=0}^1 \binom{q+u}{m-p-r}\frac{(p+r+k-1)!}{(m+k-q-u-1)!}\\
			\times&a_{p,p+k,p+r,p+r+k-1}a_{q,q+k,q+u,q+u+k-1}=0,~\forall n\in\N,
		\end{aligned}
	\end{equation}
	so
	\begin{equation}\label{mkppk}
		\sum_{r=0}^1\sum_{u=0}^1 \binom{q+u}{m-p-r}\frac{(p+r+k-1)!}{(m+k-q-u-1)!}a_{p,p+k,p+r,p+r+k-1}a_{q,q+k,q+u,q+u+k-1}=0,~\forall p,q\in\N.
	\end{equation}
	First we let $p=q$, $m=p+q+2$, then $a_{p,p+k,p+1,p+k}^2=0$, so $a_{p,p+k,p+1,p+k}=0$, $\forall p\in\N$, then \eqref{mkppk} implies $a_{p,p+k,p,p+k-1}=0$, $\forall p\in\N$. Next consider the coefficients of $x_3^mx_4^{m+2k-4}$, we have
	\begin{equation}
		\begin{aligned}
			&\sum_{p=0}^m\sum_{q=0}^{m+k}\frac{n!(n+k)!}{(n-p)!(n+k-q)!}\sum_{r=0}^2\sum_{u=0}^2 \binom{q+u}{m-p-r}\frac{(p+r+k-2)!}{(m+k-q-u-2)!}\\
			\times&a_{p,p+k,p+r,p+r+k-2}a_{q,q+k,q+u,q+u+k-2}=0,~\forall n\in\N,
		\end{aligned}
	\end{equation}
	so
	\begin{equation}\label{mkppk2}
		\sum_{r=0}^2\sum_{u=0}^2 \binom{q+u}{m-p-r}\frac{(p+r+k-2)!}{(m+k-q-u-2)!}a_{p,p+k,p+r,p+r+k-2}a_{q,q+k,q+u,q+u+k-2}=0,~\forall p,q\in\N.
	\end{equation}
	Let $p=q$, $m=p+q+4$, then $a_{p,p+k,p+2,p+k}^2=0$, so $a_{p,p+k,p+2,p+k}=0$, $\forall p\in\N$, then \eqref{mkppk2} implies
	\begin{equation}\label{mkppk21}
		\sum_{r=0}^1\sum_{u=0}^1 \binom{q+u}{m-p-r}\frac{(p+r+k-2)!}{(m+k-q-u-2)!}a_{p,p+k,p+r,p+r+k-2}a_{q,q+k,q+u,q+u+k-2}=0,~\forall p,q\in\N.
	\end{equation}
	Let $p=q$, $m=p+q+2$, then $a_{p,p+k,p+1,p+k-1}^2=0$, so $a_{p,p+k,p+1,p+k-1}=0$, $\forall p\in\N$,
	then \eqref{mkppk21} implies $a_{p,p+k,p,p+k-2}=0$, $\forall p\in\N$.
	
	Continue the above procedure, we can prove that $a_{p,p+k,u,v}=0$ where $p\leq u\leq v\leq p+k$. Conversely, by considering the coefficients of $x_3^{m+2k-2j}x_4^m$, $j=0,1,\ldots,k$, we have $a_{p,p+k,u,v}=0$ where $p\leq v\leq u\leq p+k$, which means $a_{p,p+k,u,v}=0$ where $p\leq u, v\leq p+k$.  Similarly $a_{p+k,p,u,v}=0$ where $p\leq u, v\leq p+k$. In fact we have partially proved the following conjecture:
	\begin{conjecture}\label{mfa4k}
		If $k\geq1$ and
		$$\mathcal{P}:=\sum_{p_1,p_2,p_3,p_4=0}^\infty a_{p_1,p_2,p_3,p_4}\g_1^{-p_1-p_2}\g_2^{-p_3-p_4}x_1^{p_1}x_2^{p_2}x_3^{p_3}x_4^{p_4}\in P_k(A^\infty(\R_\Theta^{4})),$$
		then $a_{p_1,p_2,p_3,p_4}=0$ where $\max_{1\leq r,s\leq 4}\n p_r-p_s\n=k$, which implies $P_k(A^\infty(\R_\Theta^{4}))=P_{k-1}(A^\infty(\R_\Theta^{4}))=\ldots=P_0(A^\infty(\R_\Theta^{4}))$.
	\end{conjecture}
	It's not hard to see that Conjecture \ref{mfam} is included in Conjecture \ref{mfa4k}, and the latter involves a couple of number sequences satisfying conditions like \eqref{mfam2} and \eqref{mfam20}. We leave the study of Conjecture \ref{mfa4k} and $P_k(A^\infty(\R_\Theta^{2n}))$ where $k\geq1$ and $n\geq3$ for future and naturally, we propose  Conjecture \ref{conjpk}. Finally, the study of $P(A^\infty(\R_\Theta^{2n}))$ where $n\geq2$, which are equivalent to the calculation of $K_0(A^\infty(\mathbb{R}^{2n}_\Theta))$ and seem much more difficult than $P(A^\infty(\R_\theta^{2}))$, are left as Conjecture \ref{k0nTh} for future study.

	\section{A representation of self-adjoint elements of $A^\infty(\R_\theta^2)$}\label{sfa}
	
	In this section we give a characterization of self-adjoint elements of $A^\infty(\R_\theta^2)$. The main result of this section is the following theorem.
	\begin{theorem}\label{sa} Let $\{B_m\}_{m\geq 0}$ be the Bernoulli number series and for $n\in\N$,
		\begin{equation}
			a_n:=\frac{(-1)^{n}(2^{2n+2}-1)B_{2n+2}}{n+1}.
		\end{equation}
		Then if \begin{equation}\label{t}
			T=\sum_{p,q=0}^\infty a_{p,q}x^py^q\in A^\infty(\R_\theta^2)
		\end{equation}
		is self-adjoint, we have
		\begin{equation}\label{pqima}
			\Im{a_{p,q}}=\sum_{k=0}^{\infty}a_{k}\binom{p+2k+1}{2k+1}\frac{(q+2k+1)!}{q!}\theta^{2k+1}\Re{a_{p+2k+1,q+2k+1}},~\forall p,q\in\N.
		\end{equation}
	\end{theorem}
	For instance, if we set $\Re a_{p,q}=e^{-p^2-q^2}$, we will get a self-adjoint element of $A^\infty(\R_\theta^2)$.
	
	Maybe it's surprising that Bernoulli numbers appear here. First let's consider the nonsmooth case first. Every element $T\in A(\R_\theta^2)$ can be written as
	\begin{equation}\label{eleinf}
		T=\sum_{k=0}^n\sum_{p+q=k}a_{p,q}x^py^q
	\end{equation}
	for some $n\in\N$. If $T=T^*$, then \eqref{mn} implies
	\begin{equation}
		\sum_{k=0}^n\sum_{p+q=k}a_{p,q}x^py^q=\sum_{k=0}^n\sum_{p+q=k}\bar{a}_{p,q}y^qx^p=\sum_{k=0}^n\sum_{p+q=k}\bar{a}_{p,q}\sum_{h=0}^{p}\binom{p}{h}\frac{q!}{(q-p+h)!}a^{p-h}x^hy^{q-p+h},
	\end{equation}
	compare the coefficients, we have
	\begin{equation}\label{fpq}
		a_{p,q}=\sum_{h=0}^{n}\binom{p+h}{h}\frac{(q+h)!}{q!}a^h\bar{a}_{p+h,q+h}.
	\end{equation}
	Here we set $a_{p,q}=0$ if $p+q>n$. For $p+q>n-2$, we have $a_{p,q}=\bar{a}_{p,q}$, which means $a_{p,q}\in\R$. And for $k\in\N$, \eqref{fpq} implies
	\begin{equation}\label{pqk}
		a_{p-k,q-k}=\bar{a}_{p-k,q-k}+\sum_{h=1}^{k}\binom{p-k+h}{h}\frac{(q-k+h)!}{(q-k)!}a^h\bar{a}_{p-k+h,q-k+h},
	\end{equation}
	$a=i\theta$ is a pure imaginary number, when $k=2l$ is even, \eqref{pqk} implies
	\begin{align*}
		0=&\sum_{h=1}^{l}\binom{p-2l+2h}{2h}\frac{(q-2l+2h)!}{(q-2l)!}a^{2h}\Re{{a}_{p-2l+2h,q-2l+2h}}\\
		&-i\sum_{h=1}^{l}\binom{p-2l+2h-1}{2h-1}\frac{(q-2l+2h-1)!}{(q-2l)!}a^{2h-1}\Im{{a}_{p-2l+2h-1,q-2l+2h-1}},\\
		2i\Im{a_{p-2l,q-2l}}=&-i\sum_{h=1}^{l}\binom{p-2l+2h}{2h}\frac{(q-2l+2h)!}{(q-2l)!}a^{2h}\Im{{a}_{p-2l+2h,q-2l+2h}}\\
		&+\sum_{h=1}^{l}\binom{p-2l+2h-1}{2h-1}\frac{(q-2l+2h-1)!}{(q-2l)!}a^{2h-1}\Re{{a}_{p-2l+2h-1,q-2l+2h-1}}.
	\end{align*}
	When $k=2l-1$ is odd, \eqref{pqk} implies
	\begin{align*}
		0=&\sum_{h=1}^{l-1}\binom{p-2l+1+2h}{2h}\frac{(q-2l+1+2h)!}{(q-2l+1)!}a^{2h}\Re{{a}_{p-2l+1+2h,q-2l+1+2h}}\\
		&-i\sum_{h=1}^{l}\binom{p-2l+2h}{2h-1}\frac{(q-2l+2h)!}{(q-2l+1)!}a^{2h-1}\Im{{a}_{p-2l+2h,q-2l+2h}},\\
		2i\Im{a_{p-2l+1,q-2l+1}}=&-i\sum_{h=1}^{l-1}\binom{p-2l+1+2h}{2h}\frac{(q-2l+1+2h)!}{(q-2l+1)!}a^{2h}\Im{{a}_{p-2l+1+2h,q-2l+1+2h}}\\
		&+\sum_{h=1}^{l}\binom{p-2l+2h}{2h-1}\frac{(q-2l+2h)!}{(q-2l+1)!}a^{2h-1}\Re{{a}_{p-2l+2h,q-2l+2h}}.
	\end{align*}
	We compute $\Im{a_{p-k,q-k}}$ for small $k$, from which we can find the general formula. First,
	\begin{equation}
		a_{p-1,q-1}=\bar{a}_{p-1,q-1}+\binom{p}{1}\frac{q!}{(q-1)!}a\bar{a}_{p,q}=\bar{a}_{p-1,q-1}+pqa{a}_{p,q},
	\end{equation}
	which means
	\begin{equation}
		\Im{a_{p-1,q-1}}=\frac{1}{2}pq\theta{a}_{p,q}.
	\end{equation}
	Also,
	\begin{align*}
		a_{p-2,q-2}&=\bar{a}_{p-2,q-2}+\binom{p-1}{1}\frac{(q-1)!}{(q-2)!}a\bar{a}_{p-1,q-1}+\binom{p}{2}\frac{q!}{(q-2)!}a^2\bar{a}_{p,q}\\
		&=\bar{a}_{p-2,q-2}+(p-1)(q-1)a\bar{a}_{p-1,q-1}+\frac{1}{2}p(p-1)q(q-1)a^2{a}_{p,q},
	\end{align*}
	take real and imaginary part of the above equality, we have
	\begin{align*}
		(p-1)(q-1)\theta\Im{{a}_{p-1,q-1}}-\frac{1}{2}p(p-1)q(q-1)\theta^2{a}_{p,q}&=0,\\
		2\Im{a_{p-2,q-2}}-(p-1)(q-1)\theta\Re{{a}_{p-1,q-1}}&=0,
	\end{align*}
	which means
	\begin{equation}
		\Im{a_{p-2,q-2}}=\frac{1}{2}(p-1)(q-1)\theta\Re{{a}_{p-1,q-1}}.
	\end{equation}
	Continue the above procedure,
	\begin{equation}
		a_{p-3,q-3}=\bar{a}_{p-3,q-3}+\binom{p-2}{1}\frac{(q-2)!}{(q-3)!}a\bar{a}_{p-2,q-2}+\binom{p-1}{2}\frac{(q-1)!}{(q-3)!}a^2\bar{a}_{p-1,q-1}+\binom{p}{3}\frac{q!}{(q-3)!}a^3\bar{a}_{p,q},
	\end{equation}
	again, take real and imaginary part of the above equality, we have
	\begin{align*}
		(p-2)(q-2)\theta\Im{{a}_{p-2,q-2}}-\binom{p-1}{2}\frac{(q-1)!}{(q-3)!}\theta^2\Re{{a}_{p-1,q-1}}&=0,\\
		2\Im{a_{p-3,q-3}}-(p-2)(q-2)\theta\Re{{a}_{p-2,q-2}}-\binom{p-1}{2}\frac{(q-1)!}{(q-3)!}\theta^2\Im{{a}_{p-1,q-1}}+\binom{p}{3}\frac{q!}{(q-3)!}\theta^3{a}_{p,q}&=0,
	\end{align*}
	so
	\begin{equation}
		\Im{a_{p-3,q-3}}=\frac{1}{2}\binom{p-2}{1}\frac{(q-2)!}{(q-3)!}\theta\Re{{a}_{p-2,q-2}}+\frac{1}{4}\binom{p}{3}\frac{q!}{(q-3)!}\theta^3{a}_{p,q}.
	\end{equation}
	For $a_{p-4,q-4}$,
	\begin{align*}
		a_{p-4,q-4}=&\bar{a}_{p-4,q-4}+\binom{p-3}{1}\frac{(q-3)!}{(q-4)!}a\bar{a}_{p-3,q-3}+\binom{p-2}{2}\frac{(q-2)!}{(q-4)!}a^2\bar{a}_{p-2,q-2}\\
		&+\binom{p-1}{3}\frac{(q-1)!}{(q-4)!}a^3\bar{a}_{p-1,q-1}+\binom{p}{4}\frac{q!}{(q-4)!}a^4\bar{a}_{p,q},
	\end{align*}
	then
	\begin{align*}
		\binom{p-3}{1}\frac{(q-3)!}{(q-4)!}\theta\Im{{a}_{p-3,q-3}}-\binom{p-2}{2}\frac{(q-2)!}{(q-4)!}\theta^2\Re{a_{p-2,q-2}}\\
		-\binom{p-1}{3}\frac{(q-1)!}{(q-4)!}\theta^3\Im{{a}_{p-1,q-1}}+\binom{p}{4}\frac{q!}{(q-4)!}\theta^4\Re{{a}_{p,q}}&=0,\\
		2\Im{a_{p-4,q-4}}-\binom{p-3}{1}\frac{(q-3)!}{(q-4)!}\theta\Re{{a}_{p-3,q-3}}-\binom{p-2}{2}\frac{(q-2)!}{(q-4)!}\theta^2\Im{{a}_{p-2,q-2}}\\
		+\binom{p-1}{3}\frac{(q-1)!}{(q-4)!}\theta^3\Re{{a}_{p-1,q-1}}&=0,
	\end{align*}
	so
	\begin{align*}
		\Im{a_{p-4,q-4}}=&\frac{1}{2}\binom{p-3}{1}\frac{(q-3)!}{(q-4)!}\theta\Re{{a}_{p-3,q-3}}+\frac{1}{2}\binom{p-2}{2}\frac{(q-2)!}{(q-4)!}\theta^2\\
		&\times\frac{1}{2}(p-1)(q-1)\theta\Re{{a}_{p-1,q-1}}-\frac{1}{2}\binom{p-1}{3}\frac{(q-1)!}{(q-4)!}\theta^3\Re{{a}_{p-1,q-1}}\\
		=&\frac{1}{2}\binom{p-3}{1}\frac{(q-3)!}{(q-4)!}\theta\Re{{a}_{p-3,q-3}}+\frac{1}{4}\binom{p-1}{3}\frac{(q-1)!}{(q-4)!}\theta^3\Re{{a}_{p-1,q-1}}.
	\end{align*}
	For $a_{p-5,q-5}$,
	\begin{align*}
		a_{p-5,q-5}=&\bar{a}_{p-5,q-5}+\binom{p-4}{1}\frac{(q-4)!}{(q-5)!}a\bar{a}_{p-4,q-4}+\binom{p-3}{2}\frac{(q-3)!}{(q-5)!}a^2\bar{a}_{p-3,q-3}\\
		&+\binom{p-2}{3}\frac{(q-2)!}{(q-5)!}a^3\bar{a}_{p-2,q-2}+\binom{p-1}{4}\frac{(q-1)!}{(q-5)!}a^4\bar{a}_{p-1,q-1}+\binom{p}{5}\frac{q!}{(q-5)!}a^5\bar{a}_{p,q},
	\end{align*}
	then
	\begin{align*}
		2\Im{a_{p-5,q-5}}=&\binom{p-4}{1}\frac{(q-4)!}{(q-5)!}\theta\Re{{a}_{p-4,q-4}}+\binom{p-3}{2}\frac{(q-3)!}{(q-5)!}\theta^2\Im{{a}_{p-3,q-3}}\\
		&-\binom{p-2}{3}\frac{(q-2)!}{(q-5)!}\theta^3\Re{{a}_{p-2,q-2}}-\binom{p-1}{4}\frac{(q-1)!}{(q-5)!}\theta^4\Im{{a}_{p-1,q-1}}\\
		&+\binom{p}{5}\frac{q!}{(q-5)!}\theta^5{a}_{p,q}\\
		=&\binom{p-4}{1}\frac{(q-4)!}{(q-5)!}\theta\Re{{a}_{p-4,q-4}}+\binom{p-3}{2}\frac{(q-3)!}{(q-5)!}\theta^2\left(\frac{1}{2}\binom{p-2}{1}\frac{(q-2)!}{(q-3)!}\right.\\
		&\left.\times\theta\Re{{a}_{p-2,q-2}}+\frac{1}{4}\binom{p}{3}\frac{q!}{(q-3)!}\theta^3{a}_{p,q}\right)-\binom{p-2}{3}\frac{(q-2)!}{(q-5)!}\theta^3\Re{{a}_{p-2,q-2}}\\
		&-\binom{p-1}{4}\frac{(q-1)!}{(q-5)!}\theta^4\times\frac{1}{2}pq\theta{a}_{p,q}+\binom{p}{5}\frac{q!}{(q-5)!}\theta^5{a}_{p,q}\\
		=&\binom{p-4}{1}\frac{(q-4)!}{(q-5)!}\theta\Re{{a}_{p-4,q-4}}+\frac{1}{2}\binom{p-2}{3}\frac{(q-2)!}{(q-5)!}\theta^3\Re{{a}_{p-2,q-2}}\\
		&+\binom{p}{5}\frac{q!}{(q-5)!}\theta^5{a}_{p,q},
	\end{align*}
	so
	\begin{align*}
		\Im{a_{p-5,q-5}}=&\frac{1}{2}\binom{p-4}{1}\frac{(q-4)!}{(q-5)!}\theta\Re{{a}_{p-4,q-4}}+\frac{1}{4}\binom{p-2}{3}\frac{(q-2)!}{(q-5)!}\theta^3\Re{{a}_{p-2,q-2}}+\frac{1}{2}\binom{p}{5}\frac{q!}{(q-5)!}\theta^5{a}_{p,q}.
	\end{align*}
	For $a_{p-6,q-6}$,
	\begin{align*}
		a_{p-6,q-6}=&\bar{a}_{p-6,q-6}+\binom{p-5}{1}\frac{(q-5)!}{(q-6)!}a\bar{a}_{p-5,q-5}+\binom{p-4}{2}\frac{(q-4)!}{(q-6)!}a^2\bar{a}_{p-4,q-4}\\
		&+\binom{p-3}{3}\frac{(q-3)!}{(q-6)!}a^3\bar{a}_{p-3,q-3}+\binom{p-2}{4}\frac{(q-2)!}{(q-6)!}a^4\bar{a}_{p-2,q-2}\\
		&+\binom{p-1}{5}\frac{(q-1)!}{(q-6)!}a^5\bar{a}_{p-1,q-1}+\binom{p}{6}\frac{q!}{(q-6)!}a^6\bar{a}_{p,q},
	\end{align*}
	then,
	\begin{align*}
		2\Im{a_{p-6,q-6}}=&\binom{p-5}{1}\frac{(q-5)!}{(q-6)!}\theta\Re{{a}_{p-5,q-5}}+\binom{p-4}{2}\frac{(q-4)!}{(q-6)!}\theta^2\Im{{a}_{p-4,q-4}}\\
		&-\binom{p-3}{3}\frac{(q-3)!}{(q-6)!}\theta^3\Re{{a}_{p-3,q-3}}-\binom{p-2}{4}\frac{(q-2)!}{(q-6)!}\theta^4\Im{{a}_{p-2,q-2}}\\
		&+\binom{p-1}{5}\frac{(q-1)!}{(q-6)!}\theta^5\Re{{a}_{p-1,q-1}}+\binom{p}{6}\frac{q!}{(q-6)!}\theta^6\Im{{a}_{p,q}}\\
		=&\binom{p-5}{1}\frac{(q-5)!}{(q-6)!}\theta\Re{{a}_{p-5,q-5}}+\binom{p-4}{2}\frac{(q-4)!}{(q-6)!}\theta^2(\frac{1}{2}\binom{p-3}{1}\frac{(q-3)!}{(q-4)!}\theta\\
		&\times\Re{{a}_{p-3,q-3}}+\frac{1}{4}\binom{p-1}{3}\frac{(q-1)!}{(q-4)!}\theta^3\Re{{a}_{p-1,q-1}})-\binom{p-3}{3}\frac{(q-3)!}{(q-6)!}\theta^3\Re{{a}_{p-3,q-3}}\\
		&-\binom{p-2}{4}\frac{(q-2)!}{(q-6)!}\theta^4\times\frac{1}{2}(p-1)(q-1)\theta\Re{{a}_{p-1,q-1}}+\binom{p-1}{5}\frac{(q-1)!}{(q-6)!}\theta^5\Re{{a}_{p-1,q-1}}\\
		=&\binom{p-5}{1}\frac{(q-5)!}{(q-6)!}\theta\Re{{a}_{p-5,q-5}}+\frac{1}{2}\binom{p-3}{3}\frac{(q-3)!}{(q-6)!}\theta^3\Re{{a}_{p-3,q-3}}\\
		&+\binom{p-1}{5}\frac{(q-1)!}{(q-6)!}\theta^5\Re{{a}_{p-1,q-1}},
	\end{align*}
	so
	\begin{align*}
		\Im{a_{p-6,q-6}}=&\frac{1}{2}\binom{p-5}{1}\frac{(q-5)!}{(q-6)!}\theta\Re{{a}_{p-5,q-5}}+\frac{1}{4}\binom{p-3}{3}\frac{(q-3)!}{(q-6)!}\theta^3\Re{{a}_{p-3,q-3}}\\
		&+\frac{1}{2}\binom{p-1}{5}\frac{(q-1)!}{(q-6)!}\theta^5\Re{{a}_{p-1,q-1}}.
	\end{align*}
	For $a_{p-7,q-7}$,
	\begin{align*}
		a_{p-7,q-7}=&\bar{a}_{p-7,q-7}+\binom{p-6}{1}\frac{(q-6)!}{(q-7)!}a\bar{a}_{p-6,q-6}+\binom{p-5}{2}\frac{(q-5)!}{(q-7)!}a^2\bar{a}_{p-5,q-5}\\
		&+\binom{p-4}{3}\frac{(q-4)!}{(q-7)!}a^3\bar{a}_{p-4,q-4}+\binom{p-3}{4}\frac{(q-3)!}{(q-7)!}a^4\bar{a}_{p-3,q-3}\\&+\binom{p-2}{5}\frac{(q-2)!}{(q-7)!}a^5\bar{a}_{p-2,q-2}+\binom{p-1}{6}\frac{(q-1)!}{(q-7)!}a^6\bar{a}_{p-1,q-1}+\binom{p}{7}\frac{q!}{(q-7)!}a^7\bar{a}_{p,q},
	\end{align*}
	then
	\begin{align*}
		2\Im{a_{p-7,q-7}}=&\binom{p-6}{1}\frac{(q-6)!}{(q-7)!}\theta\Re{{a}_{p-6,q-6}}+\binom{p-5}{2}\frac{(q-5)!}{(q-7)!}\theta^2\Im{{a}_{p-5,q-5}}\\
		&-\binom{p-4}{3}\frac{(q-4)!}{(q-7)!}\theta^3\Re{{a}_{p-4,q-4}}-\binom{p-3}{4}\frac{(q-3)!}{(q-7)!}\theta^4\Im{{a}_{p-3,q-3}}\\
		&+\binom{p-2}{5}\frac{(q-2)!}{(q-7)!}\theta^5\Re{{a}_{p-2,q-2}}+\binom{p-1}{6}\frac{(q-1)!}{(q-7)!}\theta^6\Im{{a}_{p-1,q-1}}\\
		&-\binom{p}{7}\frac{q!}{(q-7)!}\theta^7{a}_{p,q}\\
		=&\binom{p-6}{1}\frac{(q-6)!}{(q-7)!}\theta\Re{{a}_{p-6,q-6}}+\binom{p-5}{2}\frac{(q-5)!}{(q-7)!}\theta^2(\frac{1}{2}\binom{p-4}{1}\frac{(q-4)!}{(q-5)!}\\
		&\times\theta\Re{{a}_{p-4,q-4}}+\frac{1}{4}\binom{p-2}{3}\frac{(q-2)!}{(q-5)!}\theta^3\Re{{a}_{p-2,q-2}}+\frac{1}{2}\binom{p}{5}\frac{q!}{(q-5)!}\theta^5{a}_{p,q})\\
		&-\binom{p-4}{3}\frac{(q-4)!}{(q-7)!}\theta^3\Re{{a}_{p-4,q-4}}-\binom{p-3}{4}\frac{(q-3)!}{(q-7)!}\theta^4(\frac{1}{2}\binom{p-2}{1}\frac{(q-2)!}{(q-3)!}\\
		&\times\theta\Re{{a}_{p-2,q-2}}+\frac{1}{4}\binom{p}{3}\frac{q!}{(q-3)!}\theta^3{a}_{p,q})+\binom{p-2}{5}\frac{(q-2)!}{(q-7)!}\theta^5\Re{{a}_{p-2,q-2}}\\
		&+\binom{p-1}{6}\frac{(q-1)!}{(q-7)!}\theta^6\times\frac{1}{2}pq\theta{a}_{p,q}-\binom{p}{7}\frac{q!}{(q-7)!}\theta^7{a}_{p,q}\\
		=&\binom{p-6}{1}\frac{(q-6)!}{(q-7)!}\theta\Re{{a}_{p-6,q-6}}+\frac{1}{2}\binom{p-4}{3}\frac{(q-4)!}{(q-7)!}\theta^3\Re{{a}_{p-4,q-4}}\\
		&+\binom{p-2}{5}\frac{(q-2)!}{(q-7)!}\theta^5\Re{{a}_{p-2,q-2}}+\frac{17}{4}\binom{p}{7}\frac{q!}{(q-7)!}\theta^7{{a}_{p,q}},
	\end{align*}
	so
	\begin{align*}
		\Im{a_{p-7,q-7}}=&\frac{1}{2}\binom{p-6}{1}\frac{(q-6)!}{(q-7)!}\theta\Re{{a}_{p-6,q-6}}+\frac{1}{4}\binom{p-4}{3}\frac{(q-4)!}{(q-7)!}\theta^3\Re{{a}_{p-4,q-4}}\\
		&+\frac{1}{2}\binom{p-2}{5}\frac{(q-2)!}{(q-7)!}\theta^5\Re{{a}_{p-2,q-2}}+\frac{17}{8}\binom{p}{7}\frac{q!}{(q-7)!}\theta^7{{a}_{p,q}}.
	\end{align*}
	$$\ldots$$
	From the above calculation  we conjecturally set
	\begin{equation}\label{imapql}
		\Im{a_{p-2l-1,q-2l-1}}=\sum_{k=0}^la_k\binom{p-2l+2k}{2k+1}\frac{(q-2l+2k)!}{(q-2l-1)!}\theta^{2k+1}\Re{a_{p-2l+2k,q-2l+2k}},
	\end{equation}
	for a sequence $\{a_n\}_{n\geq0}$, and we can see that $a_0=\frac{1}{2}$, $a_1=\frac{1}{4}$, $a_2=\frac{1}{2}$, $a_3=\frac{17}{8}$. Note that
	\begin{align*}
		2\Im{a_{p-2l-1,q-2l-1}}=&\sum_{k=1}^{l}(-1)^{k-1}\binom{p-2l+2k-1}{2k}\frac{(q-2l+2k-1)!}{(q-2l-1)!}\theta^{2k}\Im{a_{p-2l-1+2k,q-2l-1+2k}}\\
		&+\sum_{k=0}^{l}(-1)^{k}\binom{p-2l+2k}{2k+1}\frac{(q-2l+2k)!}{(q-2l-1)!}\theta^{2k+1}\Re{a_{p-2l+2k,q-2l+2k}},
	\end{align*}
	apply \eqref{imapql} to the above equality and compare the coefficients of $\Re{a_{p-2l+2k,q-2l+2k}}$ where $0\leq k\leq l$, we have for $m\leq l$,
	\begin{equation}
		2a_m\binom{p-2l+2m}{2m+1}=\sum_{k=1}^{m}(-1)^{k-1}\binom{p-2l+2k-1}{2k}\binom{p-2l+2m}{2m-2k+1}a_{m-k}+(-1)^m\binom{p-2l+2m}{2m+1},
	\end{equation}
	after simplification, we have
	\begin{equation}\label{bna}
		a_m=\sum_{k=0}^{m}(-1)^{k+1}\binom{2m+1}{2k}a_{m-k}+(-1)^m.
	\end{equation}
	Similarly, we have
	\begin{equation}
		\Im{a_{p-2l,q-2l}}=\sum_{k=0}^{l-1}a_k\binom{p-2l+1+2k}{2k+1}\frac{(q-2l+1+2k)!}{(q-2l)!}\theta^{2k+1}\Re{a_{p-2l+1+2k,q-2l+1+2k}},
	\end{equation}
	and in summary, for $l\in\N$, conjecturally we have
	\begin{equation}\label{pqimaf}
		\Im{a_{p-l,q-l}}=\sum_{k=0}^{[\frac{l-1}{2}]}a_k\binom{p-l+1+2k}{2k+1}\frac{(q-l+1+2k)!}{(q-l)!}\theta^{2k+1}\Re{a_{p-l+1+2k,q-l+1+2k}}.
	\end{equation}
	The first few elements of $\{a_n\}_{n\geq0}$ satisfying \eqref{bna} is
	$$\frac{1}{2},\frac{1}{4},\frac{1}{2},\frac{17}{8},\frac{31}{2},\frac{691}{4},\frac{5461}{2},\frac{929569}{16},\frac{3202291}{2},\frac{221930581}{4},\frac{4722116521}{2},\ldots$$
	and with the help of the On-Line Encyclopedia of Integer Sequences (OEIS), I find that the general term formula of $\{a_n\}_{n\geq0}$ is
	\begin{equation}
		a_n=\frac{(-1)^{n}(2^{2n+2}-1)B_{2n+2}}{n+1}
	\end{equation}
	where $\{B_n\}_{n\geq 0}$ is the Bernoulli number series. We prove that this is indeed the case, even for the smooth case, i.e., \eqref{pqima} and \eqref{pqimaf} are both hold.\\
	\emph{Proof of Theorem \ref{sa}.} For \begin{equation}
		T=\sum_{p,q=0}^\infty a_{p,q}x^py^q\in A^\infty(\R_\theta^2),
	\end{equation}
	if $T=T^*$, then
	\begin{equation}
		\sum_{p,q=0}^\infty a_{p,q}x^py^q=\sum_{p,q=0}^\infty\bar{a}_{p,q}y^qx^p=\sum_{p,q=0}^\infty\bar{a}_{p,q}\sum_{h=0}^{p}\binom{p}{h}\frac{q!}{(q-p+h)!}a^{p-h}x^hy^{q-p+h},
	\end{equation}
	compare the coefficients, we have
	\begin{equation}\label{pq}
		a_{p,q}=\sum_{h=0}^{\infty}\binom{p+h}{h}\frac{(q+h)!}{q!}\bar{a}_{p+h,q+h}a^h.
	\end{equation}
	Take real and imaginary part of the above equality and note that $a=i\theta$, we have
	\begin{align*}
		\Re{a_{p,q}}=&\sum_{h=0}^\infty\binom{p+2h}{2h}\frac{(q+2h)!}{q!}a^{2h}\Re{a_{p+2h,q+2h}}-\sum_{h=0}^{\infty}\binom{p+2h+1}{2h+1}\frac{(q+2h+1)!}{q!}\\
		&\times a^{2h+1}i\Im{a_{p+2h+1,q+2h+1}}\\
		=&\sum_{h=0}^\infty(-1)^h\binom{p+2h}{2h}\frac{(q+2h)!}{q!}\theta^{2h}\Re{a_{p+2h,q+2h}}+\sum_{h=0}^{\infty}(-1)^h\binom{p+2h+1}{2h+1}\frac{(q+2h+1)!}{q!}\\
		&\times \theta^{2h+1}\Im{a_{p+2h+1,q+2h+1}},\\
		\Im{a_{p,q}}=&-\sum_{h=0}^\infty\binom{p+2h}{2h}\frac{(q+2h)!}{q!}a^{2h}\Im{a_{p+2h,q+2h}}+\frac{1}{i}\sum_{h=0}^{\infty}\binom{p+2h+1}{2h+1}\frac{(q+2h+1)!}{q!}\\
		&\times a^{2h+1}\Re{a_{p+2h+1,q+2h+1}}\\
		=&-\sum_{h=0}^\infty(-1)^h\binom{p+2h}{2h}\frac{(q+2h)!}{q!}\theta^{2h}\Im{a_{p+2h,q+2h}}+\sum_{h=0}^{\infty}(-1)^h\binom{p+2h+1}{2h+1}\frac{(q+2h+1)!}{q!}\\
		&\times\theta^{2h+1}\Re{a_{p+2h+1,q+2h+1}},
	\end{align*}
	so for $\forall k\in\N$,
	\begin{align*}
		&\sum_{h=0}^{\infty}(-1)^h\binom{p+2k+2h+1}{2h+1}\frac{(q+2k+2h+1)!}{(q+2k)!}\theta^{2h+1}\Im{a_{p+2k+2h+1,q+2k+2h+1}}\\
		=&\sum_{h=0}^\infty(-1)^{h}\binom{p+2k+2h+2}{2h+2}\frac{(q+2k+2h+2)!}{(q+2k)!}\theta^{2h+2}\Re{a_{p+2k+2h+2,q+2k+2h+2}}.
	\end{align*}
	Let $\{b_n\}_{n\geq 1}$ be a number series satisfying
	\begin{equation}\label{b}
		\sum_{k=1}^{n}(-1)^{k-1}b_k\binom{p+2n+1}{2n+1-2k}=\binom{p+2n+1}{2n+1}.
	\end{equation}
	Such a series exists. In fact, $b_1=\frac{1}{6}(p+2)(p+1)$ and for $n\geq 2$,
	\begin{equation}
		b_n=\frac{(-1)^n}{p+2n+1}\left(\binom{p+2n+1}{2n+1}-\sum_{k=1}^{n-1}(-1)^{k-1}b_k\binom{p+2n+1}{2n+1-2k}\right).
	\end{equation}
	Then
	\begin{align*}
		&(p+1)(q+1)\theta\Im a_{p+1,q+1}\\
		=&\sum_{h=0}^{\infty}(-1)^h\binom{p+2h+1}{2h+1}\frac{(q+2h+1)!}{q!}\theta^{2h+1}\Im{a_{p+2h+1,q+2h+1}}+\sum_{k=1}^{\infty}b_k\frac{(q+2k)!}{q!}\theta^{2k}\sum_{h=0}^{\infty}(-1)^h\\
		&\times\binom{p+2k+2h+1}{2h+1}\frac{(q+2k+2h+1)!}{(q+2k)!}\theta^{2h+1}\Im{a_{p+2k+2h+1,q+2k+2h+1}}\\
		=&\sum_{h=0}^\infty(-1)^{h}\binom{p+2h+2}{2h+2}\frac{(q+2h+2)!}{q!}\theta^{2h+2}\Re{a_{p+2h+2,q+2h+2}}+\sum_{k=1}^{\infty}b_k\sum_{h=0}^\infty(-1)^{h}\binom{p+2k+2h+2}{2h+2}\\
		&\times\frac{(q+2k+2h+2)!}{q!}\theta^{2k+2h+2}\Re{a_{p+2k+2h+2,q+2k+2h+2}}\\
		=&\sum_{h=0}^\infty(-1)^{h}\binom{p+2h+2}{2h+2}\frac{(q+2h+2)!}{q!}\theta^{2h+2}\Re{a_{p+2h+2,q+2h+2}}+\sum_{k=1}^{\infty}b_k\sum_{h=0}^\infty(-1)^{h}\binom{p+2k+2h+2}{2k+2h+2}\\
		&\times\frac{(2k+2h+2)!p!}{(2h+2)!(p+2k)!}\frac{(q+2k+2h+2)!}{q!}\theta^{2k+2h+2}\Re{a_{p+2k+2h+2,q+2k+2h+2}}\\
		=&\binom{p+2}{2}\frac{(q+2)!}{q!}\theta^2\Re{a_{p+2,q+2}}+\sum_{h=1}^\infty(-1)^h\left(1+\sum_{k=1}^{h}(-1)^{k}b_k\frac{(2h+2)!p!}{(2h-2k+2)!(p+2k)!}\right)\binom{p+2h+2}{2h+2}\\
		&\times\frac{(q+2h+2)!}{q!}\theta^{2h+2}\Re{a_{p+2h+2,q+2h+2}}\\
		=&\sum_{h=0}^\infty(-1)^h\left(1+\sum_{k=1}^{h}(-1)^{k}b_k\frac{(2h+2)!p!}{(2h-2k+2)!(p+2k)!}\right)\binom{p+2h+2}{2h+2}\frac{(q+2h+2)!}{q!}\theta^{2h+2}\Re{a_{p+2h+2,q+2h+2}},
	\end{align*}
	hence for $p,q\in\N$,
	\begin{equation}\label{p1q1}
		\begin{aligned}
			&\Im a_{p+1,q+1}\\
			=&\sum_{h=0}^\infty(-1)^h\left(\frac{1}{2h+2}+\sum_{k=1}^{h}(-1)^{k}b_k\frac{(2h+1)!p!}{(2h-2k+2)!(p+2k)!}\right)\binom{p+2h+2}{2h+1}\frac{(q+2h+2)!}{(q+1)!}\\
			&\times\theta^{2h+1}\Re{a_{p+2h+2,q+2h+2}}.
		\end{aligned}
	\end{equation}
	Let
	\begin{equation}
		a_h:=(-1)^h\left(\frac{1}{2h+2}+\sum_{j=1}^{h}(-1)^{j}b_j\frac{(2h+1)!p!}{(2h-2j+2)!(p+2j)!}\right),
	\end{equation}
	then for $m\in\N$,
	\begin{align*}
		&a_m-\sum_{k=0}^{m}(-1)^{k+1}\binom{2m+1}{2k}a_{m-k}\\
		=&(-1)^m\left(\frac{1}{2m+2}+\sum_{k=1}^{m}(-1)^{k}b_k\frac{(2m+1)!p!}{(2m-2k+2)!(p+2k)!}\right)-\sum_{k=0}^{m}(-1)^{m+1}\binom{2m+1}{2k}\left(\frac{1}{2m-2k+2}\right.\\
		&\left.+\sum_{j=1}^{m-k}(-1)^{j}b_j\frac{(2m-2k+1)!p!}{(2m-2k-2j+2)!(p+2j)!}\right)\\
		=&(-1)^m\left(\frac{1}{2m+2}+\sum_{k=0}^{m}\binom{2m+1}{2k}\frac{1}{2m-2k+2}\right)+(-1)^m\left(\sum_{k=1}^{m}(-1)^{k}b_k\frac{(2m+1)!p!}{(2m-2k+2)!(p+2k)!}\right.\\
		&\left.+\sum_{k=0}^m\sum_{j=1}^{m-k}(-1)^{j}b_j\binom{2m+1}{2k}\frac{(2m-2k+1)!p!}{(2m-2k-2j+2)!(p+2j)!}\right)\\
		=&\frac{(-1)^m}{2m+2}\left(1+\sum_{k=0}^{m}\binom{2m+2}{2k}\right)+(-1)^m(2m+1)!p!\left(\sum_{k=1}^{m}\frac{(-1)^{k}b_k}{(2m-2k+2)!(p+2k)!}\right.\\
		&\left.+\sum_{k=0}^m\frac{1}{(2k)!}\sum_{j=1}^{m-k}\frac{(-1)^{j}b_j}{(2m-2k-2j+2)!(p+2j)!}\right)\\
		=&\frac{(-1)^m2^{2m+1}}{2m+2}+(-1)^m(2m+1)!p!\sum_{k=1}^m\frac{(-1)^kb_k}{(p+2k)!}\left(\frac{1}{(2m-2k+2)!}+\sum_{j=0}^{m-k}\frac{1}{(2j)!(2m-2k-2j+2)!}\right)\\
		=&\frac{(-1)^m2^{2m+1}}{2m+2}+(-1)^m(2m+1)!p!\sum_{k=1}^m\frac{(-1)^k2^{2m-2k+1}b_k}{(2m-2k+2)!(p+2k)!}.
	\end{align*}
	The definition \eqref{b} of $\{b_n\}$ implies
	\begin{align*}
		2(4^n-n-1)\binom{p+2n+2}{p}&=\sum_{m=1}^n\binom{p+2n+2}{2n+1-2m}\binom{p+2m+1}{2m+1}\\
		&=\sum_{m=1}^n\binom{p+2n+2}{2n+1-2m}\sum_{k=1}^{m}(-1)^{k-1}b_k\binom{p+2m+1}{2m+1-2k}\\
		&=\sum_{k=1}^{n}(-1)^{k-1}b_k\sum_{j=1}^{n+1-k}\binom{p+2n+2}{2n-2k-2j+3}\binom{p+2k+2j-1}{2j-1}\\
		&=\sum_{k=1}^{n}\frac{(-1)^{k-1}2^{2n-2k+1}(p+2n+2)!b_k}{(2n-2k+2)!(p+2k)!},
	\end{align*}
	hence
	\begin{align*}
		&a_m-\sum_{k=0}^{m}(-1)^{k+1}\binom{2m+1}{2k}a_{m-k}\\
		=&\frac{(-1)^m2^{2m+1}}{2m+2}+(-1)^m(2m+1)!p!\sum_{k=1}^m\frac{(-1)^k2^{2m-2k+1}b_k}{(2m-2k+2)!(p+2k)!}\\
		=&\frac{(-1)^m2^{2m+1}}{2m+2}-(-1)^m(2m+1)!p!\frac{2(4^m-m-1)}{(p+2m+2)!}\binom{p+2m+2}{p}=(-1)^m.
	\end{align*}
	Finally, from \eqref{p1q1} we can see that for $h\geq1$,
	$$\Im{a_{2h,2h}}=\sum_{k=0}^{\infty}a_k\binom{2h+2k+1}{2k+1}\frac{(2h+2k+1)!}{(2h)!}\theta^{2k+1}\Re{a_{2h+2k+1,2h+2k+1}},$$
	so
	\begin{align*}
		\Im{a_{0,0}}=&-\sum_{h=0}^\infty(-1)^h\binom{2h}{2h}\frac{(2h)!}{0!}\theta^{2h}\Im{a_{2h,2h}}+\sum_{h=0}^{\infty}(-1)^h\binom{2h+1}{2h+1}\frac{(2h+1)!}{0!}\theta^{2h+1}\Re{a_{2h+1,2h+1}}\\
		=&-\Im{a_{0,0}}-\sum_{h=1}^\infty(-1)^h(2h)!\theta^{2h}\sum_{k=0}^{\infty}a_k\binom{2h+2k+1}{2k+1}\frac{(2h+2k+1)!}{(2h)!}\theta^{2k+1}\Re{a_{2h+2k+1,2h+2k+1}}\\
		&+\sum_{h=0}^{\infty}(-1)^h(2h+1)!\theta^{2h+1}\Re{a_{2h+1,2h+1}},\\
		=&-\Im{a_{0,0}}+\sum_{h=0}^{\infty}\left(a_h+\sum_{k=0}^{h}(-1)^{k+1}\binom{2h+1}{2k}a_{h-k}+(-1)^h\right)(2h+1)!\theta^{2h+1}\Re{a_{2h+1,2h+1}}\\
		=&-\Im{a_{0,0}}+2\sum_{h=0}^{\infty}a_h(2h+1)!\theta^{2h+1}\Re{a_{2h+1,2h+1}},
	\end{align*}
	hence
	\begin{equation}
		\Im{a_{0,0}}=\sum_{k=0}^{\infty}a_k(2k+1)!\theta^{2k+1}\Re{a_{2k+1,2k+1}}.
	\end{equation}
	Similarly we can prove that
	\begin{equation}
		\Im{a_{0,1}}=\sum_{k=0}^{\infty}a_k(2k+2)!\theta^{2k+1}\Re{a_{2k+1,2k+2}}
	\end{equation}
	and
	\begin{equation}
		\Im{a_{1,0}}=\sum_{k=0}^{\infty}a_k(2k+2)!\theta^{2k+1}\Re{a_{2k+2,2k+1}},
	\end{equation}
	combine with \eqref{p1q1}, we can see that for $\forall p,q\in\N$, we have
	\begin{equation}
		\Im{a_{p,q}}=\sum_{k=0}^{\infty}a_{k}\binom{p+2k+1}{2k+1}\frac{(q+2k+1)!}{q!}\theta^{2k+1}\Re{a_{p+2k+1,q+2k+1}}.
	\end{equation}
	Finally we need to determine the values of $\{a_n\}_{n\geq 0}$. Recall that the Bernoulli numbers $\{B_n\}_{n\in\N}$ are defined via
	\begin{equation}
		\frac{z}{e^z-1}+\frac{z}{2}-1=\sum_{k=1}^\infty\frac{B_{2k}}{(2k)!}z^{2k},
	\end{equation}
	then
	\begin{equation}
		\frac{2z}{e^{2z}-1}+z-1=\sum_{k=1}^\infty\frac{2^{2k}B_{2k}}{(2k)!}z^{2k},
	\end{equation}
	so
	\begin{equation}
		\sum_{k=1}^\infty\frac{(2^{2k}-1)B_{2k}}{(2k)!}z^{2k}=\frac{2z}{e^{2z}-1}+z-1-\left(\frac{z}{e^z-1}+\frac{z}{2}-1\right)=\frac{z(e^z-1)}{2(e^z+1)},
	\end{equation}
	hence
	\begin{equation}
		\left(2+\sum_{k=1}^\infty\frac{z^k}{k!}\right)\sum_{k=1}^\infty\frac{(2^{2k}-1)B_{2k}}{(2k)!}z^{2k}=\frac{1}{2}\sum_{k=1}^{\infty}\frac{z^{k+1}}{k!}.
	\end{equation}
	For $n\in\N$, compare the coefficients of $z^{2n+2}$ on both sides, we have
	\begin{equation}
		2\frac{(2^{2n+2}-1)B_{2n+2}}{(2n+2)!}+\sum_{k=1}^{n}\frac{(2^{2k}-1)B_{2k}}{(2n+2-2k)!(2k)!}=\frac{1}{2(2n+1)!},
	\end{equation}
	which means
	\begin{equation}\label{aaa}
		(2^{2n+2}-1)B_{2n+2}+\sum_{k=1}^{n+1}\binom{2n+2}{2k}(2^{2k}-1)B_{2k}=n+1.
	\end{equation}
	For $n\in\N$ let 
	\begin{equation}
		c_n:=\frac{(-1)^{n}(2^{2n+2}-1)B_{2n+2}}{n+1},
	\end{equation}
	then \eqref{aaa} implies
	\begin{equation}
		(-1)^n(n+1)c_n+\sum_{k=0}^{n}\binom{2n+2}{2k}(-1)^k(k+1)c_k=n+1,
	\end{equation}
	so
	\begin{equation}\label{ac}
		c_n-\sum_{k=0}^{n}(-1)^{k+1}\binom{2n+1}{2k}c_{n-k}=(-1)^n.
	\end{equation}
	Note that the series $\{a_n\}_{n\geq 0}$ also satisfies \eqref{ac}, so we must have 
	\begin{equation}
		a_n=c_n=\frac{(-1)^{n}(2^{2n+2}-1)B_{2n+2}}{n+1},~\forall n\in\N
	\end{equation}
	and we finish the proof.
	\qed

\end{document}